\documentclass[oneside,a4paper]{amsart}

\usepackage{amsmath,amsthm,verbatim,amssymb,amsfonts,amscd, graphicx}
\usepackage{graphics}

\usepackage[utf8]{inputenc}
\usepackage[T1]{fontenc}

\usepackage{stmaryrd}
\usepackage{xspace}
\usepackage[all]{xy}
\usepackage[margin=2.5cm]{geometry}

\usepackage{blindtext} 
\usepackage[linktocpage]{hyperref}

\theoremstyle{plain}
\newtheorem{thm}{Theorem}[section]
\newtheorem{cor}[thm]{Corollary}
\newtheorem{lem}[thm]{Lemma}
\newtheorem{prop}[thm]{Proposition}

\theoremstyle{definition}

\newtheorem*{rmq}{Remark}

\newcommand{\ho}{\mathrm{H}}

\newcommand{\Oc}{\mathcal{O}}
\newcommand{\M}{\mathcal{M}}

\newcommand{\m}{\mathfrak{m}}

\newcommand{\X}{\mathcal{X}}
\newcommand{\Y}{\mathcal{Y}}
\newcommand{\Li}{\overline{\mathcal{L}}}
\newcommand{\li}{\mathcal{L}}

\newcommand{\Q}{\mathbb{Q}}
\newcommand{\Z}{\mathbb{Z}}

\newcommand{\bbC}{\mathbb{C}}
\newcommand{\pr}{\mathbb{P}}
\newcommand{\F}{\mathbb{F}}

\newcommand{\Ker}{\mathrm{Ker}}

\newcommand{\Spec}{\mathrm{Spec}}

\usepackage[shortlabels]{enumitem}
\setlist[enumerate]{nosep}

\begin{document}

\title{
	The $\theta$-density in Arakelov geometry
	}
\author{Xiaozong WANG}
\address{Morningside Center of Mathematics, Chinese Academy of Sciences, No.55, Zhongguancun East Road, Beijing, 100190, China}
\email{xiaozong.wang@amss.ac.cn}

\begin{abstract}
	In this article, we construct a $\theta$-density for the global sections of ample Hermitian line bundles on a projective arithmetic variety. We show that this density has similar behaviour to the usual density in the Arakelov geometric setting, where only global sections of norm smaller than $1$ are considered. In particular, we prove the analogue by $\theta$-density of two Bertini kind theorems, on irreducibility and regularity respectively.
\end{abstract}

\maketitle


\section{Introduction}

The main results of this paper are Theorem \ref{mainirred} and Theorem \ref{mainvartheta}.
\subsection{Arithmetic Bertini theorems}
In the classical setting, Bertini theorems are widely used to find closed subvarieties of a projective variety sharing similar geometric properties with the original one. Here geometric properties can mean smoothness, irreducibility, connectness etc. 

Over infinite fields, such closed subvarieties can be found by choosing an embedding of the projective variety into some projective space and using the hyperplane sections to cut the variety. The infiniteness of the hyperplanes guarantees that we can always find good ones whose intersection with the projective variety suffices our need for the geometric property. We can find detailed descriptions of this case in \cite{Jo83}.

Over finite fields, we need to find good hypersurface sections rather than hyperplane sections. This is due to the fact that a projective space over a finite field possesses only finitely many hyperplanes, which makes it possible that no hyperplane has good intersection with the variety. The existence of a good hypersurface section is usually given by a density result. As we can see in \cite{Po04} by Poonen, the existence of a hypersurface whose intersection with a smooth subvariety on a projective space is smooth is guaranteed by the fact that the proportion of such hypersurfaces among all hypersurfaces of degree $d$ has a non-zero limit when $d$ tends to infinity. We also have a similar result in \cite{CP16} by Charles and Poonen for irreducibility. 
\bigskip

It is then of interest to consider the arithmetic case, i.e. to look for closed subvarieties of a projective arithmetic variety which share similar geometric properties with it. Here a projective arithmetic variety means an integral separated scheme of finite type which is flat and projective over $\Spec\ \Z$. Let $\X$ be a projective arithmetic variety of dimension $n$, and let $\Li$ be an ample Hermitian line bundle over $\X$. An ample Hermitian line bundles $\Li=(\li, \left\lVert \cdot \right\rVert)$ on $\X$ is an ample line bundle $\li$ equipped with a Hermitian metric $\left\lVert \cdot \right\rVert$ on the restriction $\li_{\bbC}$ to the fiber $\X(\bbC)$, which satisfies the following conditions :
\begin{enumerate}[label=\roman*)]
	\item $\li$ is ample over $\Spec\ \Z$ ;
	\item $\Li$ is semipositive on the complex analytic space $\X(\bbC)$ ;
	\item for any $d\gg 0$, $\ho^0(\X,\Li^{\otimes d})$ is generated by the set $\mathrm{H}^0_{\mathrm{Ar}}(\X, \Li^{\otimes d})$ of effective sections, i.e. sections of norm strictly smaller than $1$.
\end{enumerate}
This definition is given by Zhang in \cite{Zh95}. In \cite{Ch21}, Charles defined, for a subset $\mathcal{P}\subset \bigcup_{d>0} \mathrm{H}^0(\X, \Li^{\otimes d})$, the density of $\mathcal P$, when exists, as the limit
\begin{eqnarray*}
	\mu(\mathcal{P})=\lim_{d\rightarrow \infty} \frac{\#\left( \mathcal{P}\cap \mathrm{H}^0_{\mathrm{Ar}}(\X, \Li^{\otimes d}) \right)}{\#\mathrm{H}^0_{\mathrm{Ar}}(\X, \Li^{\otimes d})}.
\end{eqnarray*}
Here
\[
	\mathrm{H}^0_{\mathrm{Ar}}(\X, \Li^{\otimes d}):=\{ \sigma\in \mathrm{H}^0(\X, \Li^{\otimes d})\ ;\ \left\lVert \sigma \right\rVert<1  \}.
\]
We may call this density the \emph{Arakelov density}. In the same article, he showed that the Arakelov density of the subset
\begin{eqnarray*}
	\{ \sigma \in \bigcup_{d>0}\mathrm{H}^0(\X, \Li^{\otimes d})\ ;\ \mathrm{div}(\sigma) \text{\ is irreducible}\}
\end{eqnarray*}
is equal to $1$. 

We have a similar result on regularity in \cite{Wa22}, which is for some constant $\varepsilon>0$, the subset $\mathcal{P}_A=\bigcup_{d> 0}\mathcal{P}_{d,p\leq e^{\varepsilon d}}$, where
\[
	\mathcal{P}_{d,p\leq e^{\varepsilon d}}:=\left\{ \sigma\in \ho^0(\X,\Li^{\otimes d})\ ;\ \begin{array}{ll}
	\mathrm{div}\sigma \text{ has no singular point of residual} \\
	\text{characteristic no larger than } e^{\varepsilon d}
	\end{array}\right\},
\]
is of Arakelov density $\zeta_{\X}(1+n)^{-1}$.

\subsection{$\theta$-density}
We may also define another density of sections in $\bigcup_{d>0}\mathrm{H}^0(\X, \Li^{\otimes d})$ using the $\theta$-invariants of a Hermitian $\Z$-lattice. A Hermitian $\Z$-lattice $\overline{E}=(E, \lVert \cdot \rVert)$ is a $\Z$-lattice $E$ equipped with a Hermitian norm. See \cite{Bo20} for a detailed introduction by Bost. For a Hermitian lattice $\overline{E}$, its $\theta$-invariants are defined as
\[
	h^0_{\theta}(\overline{E}):=\log \Big(\sum_{\sigma\in \overline{E}}e^{-\pi\lVert \sigma\rVert^2}\Big),
\]
and
\[
	h^1_{\theta}(\overline{E}):=h^0_{\theta}(\overline{E}^{\vee}),
\]
where $\overline{E}^{\vee}$ is the dual lattice of $\overline{E}$.

In the setting of Arakelov geometry, where we consider a projective arithmetic variety $\X$ equipped with a Hermitian line bundle $\Li$, the $\theta$-invariants of $\Li$ are defined as
\[
	h^0_{\theta}(\X,\Li)=h^0_{\theta}\left(\ho^0(\X,\Li)\right)=\log\Big(\sum_{\sigma\in \ho^0(\X,\Li)}e^{-\pi\lVert \sigma\rVert^2}\Big),
\]
and
\[
	h^1_{\theta}(\X,\Li)=h^1_{\theta}\left(\ho^0(\X,\Li)\right)=h^0_{\theta}\left(\ho^0(\X,\Li)^{\vee}\right).
\]
Here $\ho^0(\X,\Li)$, equipped with the sup norm, is a Hermitian $\Z$-lattice. 

In Arakelov geometry, one often uses 
\[
	h_{\mathrm{Ar}}^0(\X, \Li):=\log\left( \#\mathrm{H}^0_{\mathrm{Ar}}(\X, \Li) \right)
\]
as the analogue of $h^0(X,\li)$ for a line bundle $\li$ on a projective variety $X$ over a field in the classical algebraic geometry. 
The $\theta$-invariant $h^0_{\theta}(\X,\Li)$ behaves better than $h_{\mathrm{Ar}}^0(\X, \Li)$ in various aspects. In particular it satisfies the Poisson-Riemann-Roch formula, and has good subadditivity property with respect to short exact sequences without introducing error terms (see Section 2.2.2 and Section 3.3.1 of \cite{Bo20} for these results). We define the $\theta$-density of a subset $\mathcal{P}$ of $\bigcup_{d> 0}\ho^0(\X, \Li^{\otimes d})$ following the idea of the construction of $\theta$-invariants. The idea is to consider all sections in $\bigcup_{d> 0}\ho^0(\X, \Li^{\otimes d})$ with proper weights, rather than only the effective ones. 

For a fixed ample Hermitian line bundle $\Li$, we consider a subset $\mathcal{P}=\bigcup_{d> 0}\mathcal{P}_d$ of $\bigcup_{d> 0}\ho^0(\X, \Li^{\otimes d})$, where $\mathcal{P}_d\subset \ho^0(\X, \Li^{\otimes d})$. We say that $\mathcal{P}$ has \emph{$\theta$-density} $\rho$ for some $0\leq \rho\leq 1$ if
\[
	\lim_{d\rightarrow \infty}  \frac{\sum_{\sigma\in\mathcal{P}_d}\exp(-\pi\lVert\sigma\rVert^2)}{ \sum_{\sigma \in \mathrm{H}^0(\X, \overline{\mathcal{L}}^{\otimes d})} \exp(-\pi \lVert\sigma\rVert^2) }=\rho.
\]
Then the $\theta$-density of $\mathcal{P}$ is denoted by $\mu_{\theta}(\mathcal{P})$. Similarly, we define the upper $\theta$-density $\overline{\mu_{\theta}}(\mathcal{P})$ (resp. lower $\theta$-density $\underline{\mu_{\theta}}(\mathcal{P})$) as the upper limit (resp. lower limit), when exists, of the $\theta$-proportion $\frac{\sum_{\sigma\in\mathcal{P}_d}\exp(-\pi\lVert\sigma\rVert^2)}{ \exp(h_{\theta}^0(\X, \overline{\mathcal{L}}^{\otimes d})) }$ as $d$ tends to infinity. 
\bigskip

The main theorems of this paper are the following.

\begin{thm}\label{mainirred}
	Let $\X$ be an integral projective arithmetic variety of dimension $n\geq 2$, and let $\Li$ be an ample Hermitian line bundle on $\X$. Set
	\[
		\mathcal{P}_{\mathrm{irr}}=\left\{ \sigma \in \bigcup_{d>0} \mathrm{H}^0(\X, \Li^{\otimes d})\ ; \ \mathrm{div}\sigma \ \text{is irreducible}\right\}.
	\]
	Then we have 
	\[
		\mu_{\theta}(\mathcal{P}_{\mathrm{irr}})=1.
	\]
\end{thm}

\begin{thm}\label{mainvartheta}
	Let $\X$ be a regular projective arithmetic variety of dimension $n$, and let $\Li$ be an ample line bundle on $\X$. We can find a constant $\varepsilon>0$ such that writing
	\[
		\mathcal{P}_{d,p\leq e^{\varepsilon d}}:=\left\{ \sigma\in \ho^0(\X,\Li^{\otimes d})\ ;\ \begin{array}{ll}
		\mathrm{div}\sigma \text{ has no singular point of residual} \\
		\text{characteristic smaller than } e^{\varepsilon d}
		\end{array}\right\}
	\]
	and $\mathcal{P}_A=\bigcup_{d > 0} \mathcal{P}_{d,p\leq e^{\varepsilon d}}$, we have
	\[
		\mu_{\theta}(\mathcal{P}_A)=\zeta_{\X}(1+n)^{-1}.
	\]
\end{thm}
This theorem gives a consequence on the density of sections with regular divisor :
\begin{cor}
	Let $\X$ be a regular projective arithmetic variety of dimension $n$, and let $\Li$ be an ample line bundle on $\X$. Set
	\[
		\mathcal{P}_{d}:=\left\{ \sigma\in \ho^0(\X,\Li^{\otimes d})\ ;\ 
		\mathrm{div}\sigma \text{ is regular} \right\}
	\]
	and $\mathcal{P}=\bigcup_{d> 0} \mathcal{P}_{d}$. We have
	\[
		\overline{\mu_{\theta}}(\mathcal{P})\leq\zeta_{\X}(1+n)^{-1},
	\]
	where $\overline{\mu_{\theta}}(\mathcal{P})$ is the upper $\theta$-density of $\mathcal{P}$.
\end{cor}
\begin{rmq}
	These results can be compared with their Arakelov density version, namely Theorem 1.5 in \cite{Ch21}, Theorems 1.1 and 1.2 in \cite{Wa22}. At least in the following two special cases, these two densities coincide.
\end{rmq}

It may be a frequent fact that the Arakelov density and the $\theta$-density coincide. In particular, we may see in Section \ref{restri} that for any subset $\mathcal{E}\subset \bigcup_{d> 0}\ho^0(\X,\Li^{\otimes d})$ which is the preimage of a subset $E\subset \bigcup_{d> 0}\ho^0(\X_N,\li^{\otimes d})$ over some closed subscheme $\X_N=\X\times_{\Spec\ \Z} \Spec (\Z/N\Z)$, the two densities $\mu_{Ar}(\mathcal{E})$, $\mu_{\theta}(\mathcal{E})$ of $\mathcal{E}$ always coincide with the limit 
\[
	\lim_{d\rightarrow \infty}\frac{\#\left( E\cap \ho^0(\X_N,\li^{\otimes d})\right)}{\#\ho^0(\X_N,\li^{\otimes d})}.
\]


\subsection{Previous results}
Bertini type theorems for arithmetic varieties are considered by various authors. In \cite{Mo95}, Moriwaki showed that if the projective arithmetic variety $\X$ has smooth generic fiber, then there exists a section $\sigma$ in $\ho_{\mathrm{Ar}}^0(\X,\Li^{\otimes d})$ for sufficiently large $d$ such that $\mathrm{div}\sigma_{\Q}$ is smooth. We can also find existence type results in \cite{Au01} and \cite{Au02} which says that if we permit base change of the type $g: \Spec\ \Oc_L\longrightarrow \Spec\ \Z$ where $\Oc_L$ is the ring of integers of the number field $L$, we can find a section $\sigma$ in $\ho^0(\X\times_{\Z}\Oc_L,\Li)$ with bounded height such that for any closed point $b\in \Spec\ \Oc_L$, $(\mathrm{div}\sigma)_b$ is smooth whenever $\X_{g(b)}$ is.

We can also find density type results, where the density considered is neither the Arakelov density nor the $\theta$-density. For example in \cite{Po04}, assuming the \emph{abc} conjecture, if $\X$ is a regular quasi-projective subscheme of $\pr^n_{\Z}$, Poonen showed that the density of sections $\sigma\in \bigcup_{d> 0}\ho^0(\pr^n_{\Z}, \Oc(d))$ with $\mathrm{div}\sigma\cap \X$ regular is equal to $\zeta_{\X}(1+\dim\X)^{-1}$. The density he uses relies on a choice of the system of coordinates on $\pr^n_{\Z}$ and does not involve Hermitian metric on the complex fiber.

\subsection{Strategy of proof}
We prove Theorems \ref{mainirred} and \ref{mainvartheta} following the Arakelov density version. The method is to translate the various estimates on the proportions in $\ho_{\mathrm{Ar}}^0(\X,\Li^{\otimes d})$ to the estimates on the $\theta$-proportions. The main tools are the estimates concerning the restriction maps and the estimates of $h^1_{\theta}$ in various settings.


For Theorem \ref{mainirred}, we need to treat the cases when the arithmetic variety $\X$ is an arithmetic surface and when $\dim \X>2$ separately. When $\dim \X>2$, the theorem can be deduced from the density result on one fiber $\X_p$, where we need to apply the results on restriction maps. In the surface case, we adapt the method in \cite{Ch21} on the bound of numbers on suitable decompositions of powers of line bundles on a regular arithmetic surface into product of two line bundles, and the bound on the $\theta$-proportion of sections coming from product of sections in these two line bundles. Here the bound on $h^1_{\theta}$ also plays an important role.

The proof of Theorem \ref{mainvartheta} is similar to the Arakelov density version in \cite{Wa22}, which relies on studies of proportion of bad sections on each fiber. The key to the proof is still the results on restriction maps so as to lift the proportions on the thickened fibers $\X_{p^2}$ to the $\theta$-proportions of sections on $\X$.

\subsection{Notations}

If $S$ is a finite set, we denote by $\#S$ its cardinality.
	
Let $\X$ be an arithmetic variety and $N$ a positive integer. The closed subscheme $\X\times_{\Spec\ \Z}\Spec(\Z/N\Z)$ is denoted by $\X_N$.

For an arithmetic variety $\X$ equipped with an ample Hermitian line bundle $\Li$, if $Y$ is a subscheme of $\X$ such that $Y_{\Q}\not=\emptyset$, we set $\ho^0(Y,\Li):=\ho^0(Y, \Li|_Y)$; if $Y_{\Q}$ is empty, we set $\ho^0(Y,\li):=\ho^0(Y, \li|_Y)$.

\subsection{Outline of the paper}

In Section \ref{restri}, we present the Hilbert-Samuel formula for the theta invariant $h_{\theta}^0(\X, \Li^{\otimes d}\otimes \M)$. We also prove the results on restriction morphisms that are needed for the proof of the main theorems.

Section \ref{Irred} is devoted for the proof of Theorem \ref{mainirred}, following the method in \cite{Ch21}.

Section \ref{Regular} is about the proof of Theorem \ref{mainvartheta}. We separate the closed points of $\X$ in two parts according to their residual characteristic. Applying the results proved in Section \ref{restri}, we calculate the density of sections whose divisor has singular points on each part. The proof follows the method in \cite{Wa22}.

\subsection{Acknowledgement}
We are grateful to François Charles for the interesting discussions on the subject of this article.

This project has received funding from the European Research Council (ERC) under the European Union’s Horizon 2020 research and innovation programme (grant agreement No 715747).

\section{Restriction morphisms}\label{restri}
\subsection{Hilbert-Samuel formula for $\theta$-invariants}
It is a well-known result that for an ample Hermitian line bundle $\Li$ on some projective arithmetic variety $\X$ of dimension $n$, when $d$ is sufficiently large, $\Li^{\otimes d}$ always have many global sections of small norm. Here we use the version of Charles. The following statement is Proposition 2.3 in \cite{Ch21}.
\begin{prop}\label{epsilon}
	Let $\X$ be a projective arithmetic variety, and let $\Li$ be an ample Hermitian line bundle on $\X$. Let $\overline{\mathcal{M}}=(\mathcal{M}, \lVert\cdot\rVert)$ be a Hermitian vector bundle on $\X$, and let $\overline{\mathcal{F}}$ be a coherent subsheaf of $\overline{\mathcal{M}}$. There exists a positive constant $\varepsilon_{\mathcal{F}}$ such that for any large enough integer $d$, $\ho^0(\X, \Li^{\otimes d}\otimes \overline{\mathcal{F}})\subset \ho^0(\X, \Li^{\otimes d}\otimes \overline{\mathcal{M}})$ has a basis consisting of sections with norm smaller than $e^{-\varepsilon_{\mathcal{F}}d}$.
	
	In particular, there exists $\varepsilon_0>0$ such that for any large enough integer $d$, $\ho^0(\X, \Li^{\otimes d})$ has a basis consisting of sections with norm smaller than $e^{-\varepsilon_0 d}$.
\end{prop}


The Hilbert-Samuel formula for $\theta$-invariants can be deduced from the Poisson-Riemann-Roch formula together with an estimate on $h^1_{\theta}$. The Poisson-Riemann-Roch formula over $\Spec\ \Z$ says that for any Hermitian lattice $\overline{E}=(E, \lVert\cdot\rVert)$, we have
\[
	h^0_{\theta}(\overline{E})-h^1_{\theta}(\overline{E})=\widehat{\deg}(\overline{E}),
\]
where the Arakelov degree $\widehat{\deg}(\overline{E})$ of $\overline{E}$ is equal to the logarithm of the covolume of $\overline{E}$ (see Formula (2.2.3) in \cite{Bo20}).

Under the situation as in Proposition \ref{restri}, the Poisson-Riemann-Roch formula reads
\[
	h^0_{\theta}(\X,\Li^{\otimes d}\otimes \overline{\mathcal{F}})-h^1_{\theta}(\X,\Li^{\otimes d}\otimes \overline{\mathcal{F}})=\widehat{\deg}(\Li^{\otimes d}\otimes \overline{\mathcal{F}}).
\]

In \cite{Bo20}, Bost gives an upper bound of $h^0_{\theta}(\overline{E})$ by the minimal norm of non-zero elements in the lattice $\overline{E}$, which in particular implies the following result : 
\begin{prop}\label{h1theta}
	Let $\X$ be a projective arithmetic variety of absolute dimension $n$, $\Li$ an ample Hermitian line bundle on $\X$. Let $\overline{\mathcal{M}}$ be a Hermitian vector bundle of rank $r$ on $\X$ and let $\overline{\mathcal{F}}$ be a coherent subsheaf of $\overline{\mathcal{M}}$. Let $\varepsilon_{\mathcal{F}}$ be the constant we find by Proposition \ref{epsilon}. We have
	\[
		0<h^1_{\theta}(\X, \Li^{\otimes d}\otimes \overline{\mathcal{F}})<\exp\left( h^1_{\theta}(\X, \Li^{\otimes d}\otimes \overline{\mathcal{F}}) \right)-1\leq C(\varepsilon_{\mathcal{F}}),
	\]
	where
	\[
		C(\varepsilon_{\mathcal{F}}):=3^{\mathrm{rk}\big(\ho^0(\X,\Li^{\otimes d}\otimes \overline{\mathcal{F}})\big)}\cdot\Big(\pi e^{2\varepsilon_{\mathcal{F}} d}\Big)^{-\frac{1}{2}\mathrm{rk}\big(\ho^0(\X,\Li^{\otimes d}\otimes \overline{\mathcal{F}})\big)}\int_{\pi e^{2\varepsilon_{\mathcal{F}} d}}^{+\infty}u^{\frac{1}{2}\mathrm{rk}\big(\ho^0(\X,\Li^{\otimes d}\otimes \overline{\mathcal{F}})\big)}e^{-u}\mathrm{d}u.
	\]
	In particular, when $d$ is sufficiently large, we have $e^{2\varepsilon_{\mathcal{F}} d}>\frac{1}{2\pi\cdot \mathrm{rk}\big(\ho^0(\X,\Li^{\otimes d}\otimes \overline{\mathcal{F}})\big)}$, which implies
	\begin{eqnarray*}
		C(\varepsilon_{\mathcal{F}})&\leq& 3^{\mathrm{rk}\big(\ho^0(\X,\Li^{\otimes d}\otimes \overline{\mathcal{F}})\big)}\left( 1-\frac{\mathrm{rk}\big(\ho^0(\X,\Li^{\otimes d}\otimes \overline{\mathcal{F}})\big)}{2\pi e^{2\varepsilon_{\mathcal{F}} d}} \right)^{-1}e^{-\pi e^{2\varepsilon_{\mathcal{F}} d}}\\
		&<& 2\cdot 3^{\mathrm{rk}\big(\ho^0(\X,\Li^{\otimes d}\otimes \overline{\mathcal{F}})\big)}e^{-\pi e^{2\varepsilon_{\mathcal{F}} d}}.
	\end{eqnarray*}
\end{prop}
To prove this, we need a lemma :
\begin{lem}\label{minimum}
	When $d$ is sufficiently large, any non-zero element of $\ho^0(\X,\Li^{\otimes d}\otimes\overline{\mathcal{F}})^{\vee}$ has norm larger than $e^{\varepsilon_{\mathcal{F}} d}$.
\end{lem}
\begin{proof}
	We know that $\ho^0(\X,\Li^{\otimes d}\otimes\overline{\mathcal{F}})$ is a free $\Z$-module. By Lemma \ref{epsilon}, there exists a positive constant $\varepsilon_{\mathcal{F}}$ such that for any large enough integer $d$, $\ho^0(\X, \Li^{\otimes d}\otimes\overline{\mathcal{F}})$ has a basis consisting of sections with norm smaller than $e^{-\varepsilon_{\mathcal{F}} d}$. 
	As the pairing 
	\[
		\ho^0(\X,\Li^{\otimes d}\otimes\overline{\mathcal{F}})\times \ho^0(\X,\Li^{\otimes d}\otimes\overline{\mathcal{F}})^{\vee} \longrightarrow \Z
	\]
	is perfect, there does not exist a non-zero element $\sigma^{\vee}\in \ho^0(\X,\Li^{\otimes d}\otimes\overline{\mathcal{F}})^{\vee}$  such that $\lVert \sigma^{\vee}\rVert\leq e^{\varepsilon_{\mathcal{F}} d}$. Because otherwise for any element $\sigma_i$ of the chosen basis of $\ho^0(\X,\Li^{\otimes d}\otimes\overline{\mathcal{F}})$ with norm smaller than $e^{-\varepsilon_{\mathcal{F}} d}$, we have $|(\sigma_i, \sigma^{\vee})|\leq \lVert \sigma_i\rVert\cdot \lVert \sigma^{\vee}\rVert <1$, which forces $(\sigma_i, \sigma^{\vee})$ to be $0$. But if it is true for any element in the basis, we would have $\sigma^{\vee}=0$, which leads to a contradiction. Hence any non-zero element of $\ho^0(\X,\Li^{\otimes d}\otimes\overline{\mathcal{F}})^{\vee}$ has norm larger than $e^{\varepsilon_{\mathcal{F}} d}$. 
\end{proof}
\begin{proof}[Proof of Proposition \ref{h1theta}]
	Set
	\[
		\lambda_{\mathcal{F},d}:=\min\{ \lVert \sigma\rVert ; 0\not=\sigma\in  \ho^0(\X,\Li^{\otimes d}\otimes \overline{\mathcal{F}})^{\vee} \}.
	\]
	We know by Lemma \ref{minimum} that when $d$ is large enough, $\lambda_{\mathcal{F},d}>e^{\varepsilon_{\mathcal{F}} d}$. Applying \cite[Prop 2.6.2]{Bo20} to the Hermitian lattice $\ho^0(\X,\Li^{\otimes d}\otimes \overline{\mathcal{F}})^{\vee}$, we get
	\[
		0<h^1_{\theta}(\X, \Li^{\otimes d}\otimes \overline{\mathcal{F}}) < \exp\left( h^1_{\theta}(\X, \Li^{\otimes d}\otimes \overline{\mathcal{F}}) \right)-1\leq C(\mathrm{rk}\big(\ho^0(\X,\Li^{\otimes d}\otimes \overline{\mathcal{F}})\big), \lambda_{\mathcal{F}, d}),
	\]
	where
	\begin{eqnarray*}
		& &C(\mathrm{rk}\big(\ho^0(\X,\Li^{\otimes d}\otimes \overline{\mathcal{F}})\big), \lambda_{\mathcal{F}, d})\\
		&=&3^{\mathrm{rk}\big(\ho^0(\X,\Li^{\otimes d}\otimes \overline{\mathcal{F}})\big)}\big(\pi \lambda_{\mathcal{F}, d}^{2}\big)^{-\frac{1}{2}\mathrm{rk}\big(\ho^0(\X,\Li^{\otimes d}\otimes \overline{\mathcal{F}})\big)}\int_{\pi \lambda_{\mathcal{F}, d}^{2}}^{+\infty}u^{\frac{1}{2}\mathrm{rk}\big(\ho^0(\X,\Li^{\otimes d}\otimes \overline{\mathcal{F}})\big)}e^{-u}\mathrm{d}u\\
		&<&C(\varepsilon_{\mathcal{F}}).
	\end{eqnarray*}
	Then the inequality
	\begin{eqnarray*}
		C(\varepsilon_{\mathcal{F}})&\leq& 3^{\mathrm{rk}\big(\ho^0(\X,\Li^{\otimes d}\otimes \overline{\mathcal{F}})\big)}\left( 1-\frac{\mathrm{rk}\big(\ho^0(\X,\Li^{\otimes d}\otimes \overline{\mathcal{F}})\big)}{2\pi e^{2\varepsilon_{\mathcal{F}} d}} \right)^{-1}e^{-\pi e^{2\varepsilon_{\mathcal{F}} d}}\\
		&<& 2\cdot 3^{\mathrm{rk}\big(\ho^0(\X,\Li^{\otimes d}\otimes \overline{\mathcal{F}})\big)}e^{-\pi e^{2\varepsilon_{\mathcal{F}} d}}.
	\end{eqnarray*}
	is a consequence of Lemma 2.6.4 in \cite{Bo20}. Hence we conclude.
\end{proof}
When $d$ is large enough, we have
\[
	\mathrm{rk}\big(\ho^0(\X,\Li^{\otimes d}\otimes \overline{\mathcal{M}})\big)=h^0(\X_{\Q},\Li^{\otimes d}\otimes \overline{\mathcal{M}}).
\]
As $\dim\X_{\Q}=n-1$, the assymptotic Riemann-Roch theorem implies that we can find a constant $K>0$ such that
\[
	\mathrm{rk}\big(\ho^0(\X,\Li^{\otimes d}\otimes \overline{\mathcal{F}})\big)\leq \mathrm{rk}\big(\ho^0(\X,\Li^{\otimes d}\otimes \overline{\mathcal{M}})\big)<K d^{n-1}.
\]
Then we have the following corollary : 
\begin{cor}\label{h1limit}
	Under the same situation as Proposition \ref{h1theta}, there exists a constant $K>0$ such that when $d$ is large enough,
	\[
		h^1_{\theta}(\X, \Li^{\otimes d}\otimes \overline{\mathcal{F}})=O\left( \exp\left( Kd^{n-1}-\pi e^{2\varepsilon_{\mathcal{F}} d} \right) \right).
	\]
	In particular, we have
	\[
		\lim_{d\rightarrow \infty} h^1_{\theta}(\X, \Li^{\otimes d}\otimes \overline{\mathcal{F}})=0.
	\]
\end{cor}
Note that for a Hermitian vector bundle $\overline{\mathcal{M}}$, Zhang proved in \cite{Zh95} an estimate of the Arakelov degree
\[
	\widehat{\deg}(\Li^{\otimes d}\otimes \overline{\mathcal{M}})=\frac{r}{n!}\Li^{n} d^n+o(d^n).
\]
The estimate is given in Theorem (1.4) of the paper, where the Arakelov degree $\widehat{\deg}(\Li^{\otimes d}\otimes \overline{\mathcal{M}})$ is denoted by $\chi_{\mathrm{sup}}(\Li^{\otimes d}\otimes \overline{\mathcal{M}})$. The above estimate together with the Poisson-Riemann-Roch formula gives us the $\theta$-version of the Hilbert-Samuel formula as follows :
\begin{thm}\label{hilbsamutheta}
	Let $\X$ be a projective arithmetic variety of absolute dimension $n$, $\Li$ an ample Hermitian line bundle on $\X$ and $\overline{\mathcal{M}}$ a Hermitian vector bundle of rank $r$ on $\X$. Then as $d$ tends to $\infty$, we have
	\[
		h^0_{\theta}(\X,\Li^{\otimes d}\otimes \overline{\mathcal{M}})= \frac{r}{n!}\Li^{n} d^n+o(d^n).
	\]
\end{thm}

\subsection{Results on restriction morphisms}\label{restriction}
We give two estimates concerning restriction morphisms. The first one is about the $\theta$-proportion of global sections of $\Li^{\otimes d}$ which vanishes on a fixed closed subscheme. We show that this $\theta$-proportion tends to $0$ when $d$ tends to infinity. The second one estimates the $\theta$-proportion of global sections vanishing on $\X_N$ for a positive integer $N$, especially how this proportion behaves when we change the integer $N$.
\begin{prop} \label{fibre}
	Let $\Y$ be a closed subscheme of $\X$ such that the generic fiber $\Y_{\Q}$ of $\Y$ is reduced. Let
	\[
		\phi_{d, \Y}: \mathrm{H}^0(\X, \Li^{\otimes d} ) \longrightarrow \mathrm{H}^0(\Y, \Li^{\otimes d} ) 
	\]
	be the natural restriction morphism. Then there exists a positive integer $d_0$ and a constant $\eta \in \mathbb{R}_+$ which are independ of $\Y$ such that for any $d\geq  d_0$, 
	\[
		\frac{\sum_{\sigma\in \mathrm{Ker}(\phi_{d, \Y}) } \exp{(-\pi||\sigma||^{2})}}{\sum_{\sigma\in \mathrm{H}^0(\X, \Li^{\otimes d}) } \exp{(-\pi||\sigma||^{2})}}
		= O\left(\exp{(-\eta d^{\mathrm{dim} \Y})}\right),
	\]
	where the constants involved in the estimate are also independent of $\Y$.
\end{prop}

\begin{proof}
	We may assume that $\Y$ is reduced and irreducible. 
	
	We first deal with the case when $\Y$ is horizontal, i.e. $\Y$ is flat over $\Spec\ \Z$. Proposition 2.20 in \cite{Ch21} applied to $\X_{\Q}$ and $\Y_{\Q}$ tells us that we can find constants $d_0\in \Z_{>0}$ and $C$ independent of $\Y$ such that for any $d\geq d_0$, the image of the restriction 
	\[
		\phi_{d, \Y}: \mathrm{H}^0(\X, \Li^{\otimes d} ) \longrightarrow \mathrm{H}^0(\Y, \Li^{\otimes d} )
	\]
	has $\Z$-rank at least $Cd^{\dim \Y-1}$. We denote by $k_d$ the $\Z$-rank of the image. So $k_d\geq Cd^{\dim \Y-1}$ when $d\geq d_0$. Enlarging $d_0$ if needed,  for any $d\geq d_0$ we can find a basis of $\mathrm{H}^0(\X, \Li^{\otimes d} )$ consisting of sections of norm smaller than $e^{-\varepsilon_0 d}$ by Proposition \ref{epsilon}. We may assume that the first $k_d$ elements of this basis, say $\sigma_1,\dots, \sigma_{k_d}$, have linearly independent images in $\mathrm{H}^0(\Y, \Li^{\otimes d} )$. 
	
	Now choose $\eta<\varepsilon_0$. Note that $\mathrm{H}^0(\X, \Li^{\otimes d} )=\Ker \phi_{d, \Y} \oplus \left( \bigoplus_{1\leq i\leq k_d} \Z\sigma_i \right)$. For any $\sigma\in \Ker \phi_{d, \Y} $, any integers $\lambda_i$ satisfying $|\lambda_i|<e^{\eta d}$ for $1\leq i\leq k_d$, by the parallelogram law we have
	\begin{eqnarray*}
		2\left\lVert \sigma+\sum_{1\leq i\leq k_d}\lambda_i\sigma_i \right\rVert^2+2\left\lVert (-\sigma)+\sum_{1\leq i\leq k_d}\lambda_i\sigma_i \right\rVert^2
		&=& \lVert 2\sigma \rVert^2+\left\lVert 2\sum_{1\leq i\leq k_d}\lambda_i\sigma_i \right\rVert^2 \\
		&\leq & 4\lVert \sigma \rVert^2  +4 \sum_{1\leq i\leq k_d}\lambda_i^2\lVert \sigma_i \rVert^2\\
		&\leq & 4\lVert \sigma \rVert^2+ 4k_d e^{2(\eta-\varepsilon_0)d}.
	\end{eqnarray*}
	Using this inequality, we obtain
	\begin{eqnarray*}
		\exp\Big(h^0_{\theta}(\X, \Li^{\otimes d} )\Big) &\geq&  \sum_{|\lambda_i|<e^{\eta d},1\leq i\leq k_d} \left(\sum_{\sigma \in \Ker \phi_{d, \Y} } e^{-\pi\lVert \sigma+\sum_{i}\lambda_i\sigma_i \rVert^2}\right) \\
		&=&  \sum_{|\lambda_i|<e^{\eta d},1\leq i\leq k_d} \left(\frac{1}{2}\sum_{\sigma \in \Ker \phi_{d, \Y} } \Big(e^{-\pi\lVert \sigma+\sum_{i}\lambda_i\sigma_i \rVert^2}+e^{-\pi\lVert (-\sigma)+\sum_{i}\lambda_i\sigma_i \rVert^2}\Big)\right) \\
		&\geq&  \sum_{|\lambda_i|<e^{\eta d},1\leq i\leq k_d} \left(\sum_{\sigma \in \Ker \phi_{d, \Y} } e^{-\frac{1}{2}\pi(\lVert \sigma+\sum_{i}\lambda_i\sigma_i \rVert^2+\lVert (-\sigma)+\sum_{i}\lambda_i\sigma_i \rVert^2)}\right) \\
		&\geq&  \sum_{|\lambda_i|<e^{\eta d},1\leq i\leq k_d} \left(\sum_{\sigma \in \Ker \phi_{d, \Y} } e^{-\frac{1}{2}\pi(2\lVert \sigma \rVert^2+ 2k_d e^{2(\eta-\varepsilon_0)d})}\right) \\
		&=& \exp\Big(h^0_{\theta}(\Ker \phi_{d, \Y}) \Big) \sum_{|\lambda_i|<e^{\eta d},1\leq i\leq k_d} e^{-\pi k_d e^{2(\eta-\varepsilon_0) d}}\\
		&\geq & \exp\Big(h^0_{\theta}(\Ker \phi_{d, \Y}) \Big)\cdot e^{\eta k_d  d} e^{-\pi k_d e^{2(\eta-\varepsilon_0) d}} ,
	\end{eqnarray*}
	where all the $\lambda_i$'s in the sum are integers. 
	
	As $k_d\leq h^0(\X_{\Q}, \Li^{\otimes d} )$ and that the chosen $\eta$ is strictly smaller than $\varepsilon_0$, up to enlarging $d_0$, we may assume that when $d\geq d_0$, we have 
	\[
		\frac{1}{2}<e^{-\pi h^0(\X_{\Q}, \Li^{\otimes d})\cdot e^{2(\eta-\varepsilon_0) d}} <1.
	\]
	Therefore $e^{-\pi k_d e^{2(\eta-\varepsilon_0) d}}>\frac{1}{2}$, and we deduce that when $d\geq d_0$,
	\[
		\exp\Big(h^0_{\theta}(\X, \Li^{\otimes d} )\Big) >\frac{1}{2}\exp\Big(h^0_{\theta}(\Ker \phi_{d, \Y}) \Big)\cdot e^{\eta k_d  d}.
	\]
	Since $k_d\geq Cd^{\dim \Y-1}$ when $d\geq d_0$, we then obtain the conclusion
	\[
		\frac{\sum_{\sigma\in \mathrm{Ker}(\phi_{d, \Y}) } \exp{(-\pi||\sigma||^{2})}}{\sum_{\sigma\in \mathrm{H}^0(\X, \Li^{\otimes d}) } \exp{(-\pi||\sigma||^{2})}}
		=\exp\Big( h^0_{\theta}(\Ker \phi_{d, \Y})- h^0_{\theta}(\X, \Li^{\otimes d} )\Big)
		= O(\exp{(-\eta d^{\mathrm{dim} \Y})}),
	\]
	where the constant involved in is independent of $\Y$.
	\medskip
	
	Now we deal with the case when $\Y$ is vertical. By our assumption, $\Y$ is reduced and irreducible. In this case $\Y$ is actually supported on the subscheme $\X_p$ for some prime number $p$. Again by Proposition 2.20 in \cite{Ch21}, we can find constants $d_0\in \Z_{>0}$ and $C$ independent of $\Y$ so that for any $d\geq d_0$, the image of the restriction 
	\[
		\phi_{d, \Y}: \mathrm{H}^0(\X, \Li^{\otimes d} ) \longrightarrow \mathrm{H}^0(\Y, \li^{\otimes d} )
	\]
	has $\F_p$-dimension at least $Cd^{\dim \Y}$. In other words, the cardinality of the image is at least $p^{Cd^{\dim\Y}}$. We still denote by $k_d$ the $\F_p$-dimension of the image. Then still by Proposition \ref{epsilon}, up to enlarging $d_0$, we may find $\sigma_1,\dots,\sigma_{k_d}\in  \mathrm{H}^0(\X, \Li^{\otimes d} ) $ with $\lVert \sigma_i\rVert <e^{-\varepsilon_0 d}$ for each $1\leq i\leq k_d$ and that their images in $\mathrm{H}^0(\Y, \li^{\otimes d} )$ are $\F_p$-linearly independent. In this case, any section $\sigma\in \mathrm{H}^0(\X, \Li^{\otimes d} )$ can be uniquely decomposed to the form
	\[
		\sigma=\sigma'+\sum_{1\leq i\leq k_d}\lambda_i\sigma_i,
	\]
	where $\sigma'\in \Ker \phi_{d, \Y} $, and for each $1\leq i\leq k_d$, $0\leq \lambda_i<p$ is an integer. Therefore when $d\geq d_0$, by the same method as above,  we have
	\begin{eqnarray*}
		\exp\Big(h^0_{\theta}(\X, \Li^{\otimes d} )\Big) &=&  \sum_{0\leq\lambda_i<p,1\leq i\leq k_d} \left(\sum_{\sigma \in \Ker \phi_{d, \Y} } e^{-\pi\lVert \sigma+\sum_{i}\lambda_i\sigma_i \rVert^2}\right) \\
		&\geq&  \sum_{0\leq\lambda_i<p,1\leq i\leq k_d} \left(\sum_{\sigma \in \Ker \phi_{d, \Y} } e^{-\frac{1}{2}\pi(2\lVert \sigma \rVert^2+ 2\sum_{i}\lambda_i^2\lVert\sigma_i \rVert^2)}\right) \\
		&\geq&  \sum_{0\leq\lambda_i<1,1\leq i\leq k_d}  \left(\sum_{\sigma \in \Ker \phi_{d, \Y} } e^{-\pi\lVert \sigma \rVert^2}\prod_{i}e^{-\pi \lambda_i^2\lVert\sigma_i \rVert^2}\right)\\
		&=& \exp\Big(h^0_{\theta}(\Ker \phi_{d, \Y}) \Big)\sum_{0\leq\lambda_i\leq 1,1\leq i\leq k_d}  \left(\prod_{i}e^{-\pi \lambda_i^2\lVert\sigma_i \rVert^2}\right)\\
		&\geq & \exp\Big(h^0_{\theta}(\Ker \phi_{d, \Y}) \Big)\sum_{0\leq\lambda_i\leq 1,1\leq i\leq k_d} e^{-\pi k_d e^{-\varepsilon_0 d}}\\
		&\geq & \exp\Big(h^0_{\theta}(\Ker \phi_{d, \Y}) \Big)\cdot 2^{k_d  } e^{-\pi k_de^{-\varepsilon_0 d}}.
	\end{eqnarray*}
	Up to enlarging $d_0$, we may assume that for any $d\geq d_0$,
	\[
		\frac{1}{2}<e^{-\pi h^0(\X_{\Q}, \Li^{\otimes d})\cdot e^{-\varepsilon_0 d}}<1.
	\]
	So for such $d$, we have $e^{-\pi k_de^{-\varepsilon_0 d}}>\frac{1}{2}$ and that
	\[
		\exp\Big(h^0_{\theta}(\X, \Li^{\otimes d} )\Big)> \frac{1}{2}\exp\Big(h^0_{\theta}(\Ker \phi_{d, \Y}) \Big)\cdot 2^{k_d},
	\]
	which, when taking $\eta<\log 2$, induces
	\begin{eqnarray*}
		\frac{\sum_{\sigma\in \mathrm{Ker}(\phi_{d, \Y}) } \exp{(-\pi||\sigma||^{2})}}{\sum_{\sigma\in \mathrm{H}^0(\X, \Li^{\otimes d}) } \exp{(-\pi||\sigma||^{2})}}
		&=& \exp\Big( h^0_{\theta}(\Ker \phi_{d, \Y})- h^0_{\theta}(\X, \Li^{\otimes d} )\Big) \\
		&<& 2e^{-k_d\log 2}\\
		&=& O(\exp{(-\eta d^{\mathrm{dim} \Y})}),
	\end{eqnarray*}
	as $k_d\geq Cd^{\dim \Y}$. Here the constant involved in is also independent of $\Y$.
\end{proof}
The following lemma is a basic property we use to relate the restrictions to fibers with the modulo $N$ map. See Lemma 2.3 in \cite{Wa22} for a proof.
\begin{lem}\label{freemod}
	Let $\X$ be a projective arithmetic variety, and let $\Li$ be an ample Hermitian line bundle on $\X$. There is a positive integer $d_0$ such that when $d\geq d_0$, we have 
	\[
		 \ho^0(\X_N,\li^{\otimes d})\simeq\ho^0(\X,\li^{\otimes d})/\big(N\cdot \ho^0(\X,\li^{\otimes d})\big)
	\]
	for any positive integer $N$.
\end{lem}

Now we present the second estimate.
\begin{prop}\label{restheta}
	Let $\X$ be a projective arithmetic variety, and let $\Li$ be an ample Hermitian line bundle on $\X$. For a positive integer $N$, consider the restriction modulo $N$ map
	\[
		\phi_{d,N}: \ho^0(\X,\Li^{\otimes d}) \longrightarrow \ho^0(\X_{N},\li^{\otimes d}).
	\]
	For any $\sigma_0\in \ho^0(\X,\Li^{\otimes d})$, we have
	\begin{equation}
		\frac{\sum_{\sigma\in \mathrm{Ker}(\phi_{d,N})} \exp(-\pi \lVert\sigma+\sigma_0\rVert^2) }{ \sum_{\sigma \in \mathrm{H}^0(\X, \Li^{\otimes d})} \exp(-\pi \lVert\sigma\rVert^2) }
		< \frac{\exp\left(h^1_{\theta}(\mathrm{Ker}(\phi_{d,N}))\right)}{\#\ho^0(\X_{N},\li^{\otimes d})},\label{2.10.1}
	\end{equation}
	and
	\begin{equation}
		\frac{\sum_{\sigma\in \mathrm{Ker}(\phi_{d,N})} \exp(-\pi \lVert\sigma+\sigma_0\rVert^2) }{ \sum_{\sigma \in \mathrm{H}^0(\X, \Li^{\otimes d})} \exp(-\pi \lVert\sigma\rVert^2) }
		>\frac{\left[2\cdot\exp\left(-h^1_{\theta}(\X, \Li^{\otimes d} )\right)- \exp\left(h^1_{\theta}(\mathrm{Ker} \phi_{d,N} )\right) \right]}{\#\ho^0(\X_{N},\li^{\otimes d})}. \label{2.10.2}
	\end{equation}
	In particular, if $N$ is fixed, we have for any $\sigma_0\in \ho^0(\X,\Li^{\otimes d})$,
	\[
		\lim_{d\rightarrow \infty}\frac{\sum_{\sigma\in \mathrm{Ker}(\phi_{d,N})} \exp(-\pi \lVert\sigma+\sigma_0\rVert^2) }{ \sum_{\sigma \in \mathrm{H}^0(\X, \Li^{\otimes d})} \exp(-\pi \lVert\sigma\rVert^2) } \cdot \left(\#\ho^0(\X_{N},\li^{\otimes d})\right)=1.
	\]
\end{prop}
\begin{proof}
	By Lemma \ref{freemod}, we may choose a positive integer $d_0$ such that when $d\geq d_0$, for any positive integer $N$, the restriction morphism
	\[
		\phi_{d,N}: \ho^0(\X,\Li^{\otimes d}) \longrightarrow \ho^0(\X_{N},\li^{\otimes d})
	\]
	is surjective.
	
	We start with the proof of the inequality \eqref{2.10.1}. First assume that $\sigma_0=0$ in $\ho^0(\X,\Li^{\otimes d})$. In this case we have
	\[
		\frac{\sum_{\sigma\in \mathrm{Ker}(\phi_{d,N})} \exp(-\pi \lVert\sigma\rVert^2) }{ \sum_{\sigma \in \mathrm{H}^0(\X, \overline{\mathcal{L}}^{\otimes d})} \exp(-\pi \lVert\sigma\rVert^2) }=\exp\left( h^0_{\theta}(\mathrm{Ker}(\phi_{d,N}))-h^0_{\theta}(\X, \Li^{\otimes d}) \right).
	\]
	By the Poisson-Riemann-Roch formula, we get
	\[
		h^0_{\theta}(\mathrm{Ker}(\phi_{d,N})) = \widehat{\deg}\mathrm{Ker}(\phi_{d,N}) +h^1_{\theta}(\mathrm{Ker}(\phi_{d,N})).
	\]
	On the other hand, as $\mathrm{Ker}(\phi_{d,N})$ is a sub-lattice of $\ho^0(\X, \Li^{\otimes d})$ of index 
	\[
		\#\ho^0(\X_{N},\li^{\otimes d})=N^{\mathrm{rk}\big(\ho^0(\X,\Li^{\otimes d})\big)},
	\]
	we have
	\begin{eqnarray*}
		\widehat{\deg}\mathrm{Ker}(\phi_{d,N}) &=&  \widehat{\deg}\overline{\mathcal L}^{\otimes d} -\mathrm{rk}\big(\ho^0(\X,\Li^{\otimes d})\big)\cdot\log N \\
		&=& h^0_{\theta}(\X, \Li^{\otimes d} )-h^1_{\theta}(\X, \Li^{\otimes d} ) -\mathrm{rk}\big(\ho^0(\X,\Li^{\otimes d})\big)\cdot\log N.
	\end{eqnarray*}
	This induces
	\[
		h^0_{\theta}(\mathrm{Ker}(\phi_{d,N})) - h^0_{\theta}(\X, \Li^{\otimes d} )=h^1_{\theta}(\mathrm{Ker}(\phi_{d,N}))-\mathrm{h}^1_{\theta}(\X, \Li^{\otimes d} ) -\mathrm{rk}\big(\ho^0(\X,\Li^{\otimes d})\big)\cdot\log N.
	\]
	Therefore we have
	\[
		-h^1_{\theta}(\X, \Li^{\otimes d} )<h^0_{\theta}(\mathrm{Ker}(\phi_{d,N})) - h^0_{\theta}(\X, \Li^{\otimes d} )+\mathrm{rk}\big(\ho^0(\X,\Li^{\otimes d})\big)\cdot\log N<h^1_{\theta}(\mathrm{Ker}(\phi_{d,N})),
	\]
	which implies
	\[
		\frac{\sum_{\sigma\in \mathrm{Ker}(\phi_{d,N})} \exp(-\pi \lVert\sigma\rVert^2) }{ \sum_{\sigma \in \mathrm{H}^0(\X, \overline{\mathcal{L}}^{\otimes d})} \exp(-\pi \lVert\sigma\rVert^2) }
		< \frac{\exp\left(h^1_{\theta}(\mathrm{Ker}(\phi_{d,N}))\right)}{\#\ho^0(\X_{N},\li^{\otimes d})}.
	\]
	By the formula (2.1.7) in \cite{Bo20}, for any $\sigma_0\in \ho^0(\X,\Li^{\otimes d})$, we have
	\begin{eqnarray*}
		\frac{\sum_{\sigma\in \mathrm{Ker}(\phi_{d,N})} \exp(-\pi \lVert\sigma+\sigma_0\rVert^2) }{ \sum_{\sigma \in \mathrm{H}^0(\X, \overline{\mathcal{L}}^{\otimes d})} \exp(-\pi \lVert\sigma\rVert^2) } &\leq& \frac{\sum_{\sigma\in \mathrm{Ker}(\phi_{d,N})} \exp(-\pi \lVert\sigma\rVert^2) }{ \sum_{\sigma \in \mathrm{H}^0(\X, \overline{\mathcal{L}}^{\otimes d})} \exp(-\pi \lVert\sigma\rVert^2) } \\
		&<& \frac{\exp\left(h^1_{\theta}(\mathrm{Ker}(\phi_{d,N}))\right)}{\#\ho^0(\X_{N},\li^{\otimes d})}.
	\end{eqnarray*}
	This finish the proof of the inequality \eqref{2.10.1}.\\
	
	The inequality \eqref{2.10.2} is a consequence of the formula (2.1.8) in \cite{Bo20}. In fact, this formula tells us that
	\begin{eqnarray*}
		\sum_{\sigma\in \mathrm{Ker}(\phi_{d,N})} \exp(-\pi ||\sigma+\sigma_0||^2)+\sum_{\sigma\in \mathrm{Ker}(\phi_{d,N})} \exp(-\pi ||\sigma||^2)\geq 2\exp\left(\widehat{\deg} \mathrm{Ker}(\phi_{d,N})\right).
	\end{eqnarray*}
	We can transform the right side to
	\begin{eqnarray*}
		2\exp\left(\widehat{\deg} \mathrm{Ker}(\phi_{d,N})\right) &=& 2\exp\left( \widehat{\deg}\overline{\mathcal L}^{\otimes d} - \mathrm{rk}\big(\ho^0(\X,\Li^{\otimes d})\big)\cdot\log N \right) \\
		&=& 2\cdot \frac{\exp\left( h^0_{\theta}(\X, \overline{\mathcal L}^{\otimes d} )-h^1_{\theta}(\X, \overline{\mathcal L}^{\otimes d} )\right)}{\#\ho^0(\X_{N},\li^{\otimes d})} \\
		&=& 2\cdot \frac{\exp\left( h^0_{\theta}(\X, \overline{\mathcal L}^{\otimes d} )\right)}{\#\ho^0(\X_{N},\li^{\otimes d})}  \cdot \exp\left(-h^1_{\theta}(\X, \overline{\mathcal L}^{\otimes d} )\right).
	\end{eqnarray*}
	This suggests that
	\begin{eqnarray*}
		& &\sum_{\sigma\in \mathrm{Ker}(\phi_{d,N})} \exp(-\pi ||\sigma+\sigma_0||^2)\\
		&\geq&  2\exp\left(\widehat{\deg} \mathrm{Ker}(\phi_{d,N})\right)-\sum_{\sigma\in \mathrm{Ker}(\phi_{d,N})} \exp(-\pi ||\sigma||^2)\\
		&>& 2\cdot \frac{\exp\left( h^0_{\theta}(\X, \overline{\mathcal L}^{\otimes d} )\right)}{\#\ho^0(\X_{N},\li^{\otimes d})}  \cdot \exp\left(-h^1_{\theta}(\X, \overline{\mathcal L}^{\otimes d} )\right)-\frac{\exp\left( h^0_{\theta}(\X, \overline{\mathcal L}^{\otimes d} )\right)}{\#\ho^0(\X_{N},\li^{\otimes d})}\exp\left(h^1_{\theta}(\mathrm{Ker}(\phi_{d,N}))\right),
	\end{eqnarray*}
	which proves the inequality (\ref{2.10.2}).

	Note that Corollary \ref{h1limit} applied to the trivial case $\overline{\mathcal{F}}=\overline{\mathcal{O}_{\X}}$, where $\overline{\mathcal{O}_{\X}}$ is equipped with the trivial metric, implies that $\lim_{d\rightarrow \infty}h^1_{\theta}(\X, \overline{\mathcal L}^{\otimes d} )=0$. If we take $\overline{\mathcal{F}}=N\cdot \overline{\mathcal{O}_{\X}}\subset \overline{\mathcal{O}_{\X}}$, we get $\ho^0(\X,\Li^{\otimes d}\otimes\overline{\mathcal{F}})=\mathrm{Ker}(\phi_{d,N})$. Then by Corollary \ref{h1limit} we also have 
	\[
		\lim_{d\rightarrow \infty}h^1_{\theta}(\mathrm{Ker}(\phi_{d,N}))=0.
	\]
	Therefore when $N$ is fixed, for any $\sigma_0\in \ho^0(\X,\Li^{\otimes d})$, the limit
	\[
		\lim_{d\rightarrow \infty}\frac{\sum_{\sigma\in \mathrm{Ker}(\phi_{d,N})} \exp(-\pi \lVert\sigma+\sigma_0\rVert^2) }{ \sum_{\sigma \in \mathrm{H}^0(\X, \Li^{\otimes d})} \exp(-\pi \lVert\sigma\rVert^2) } \cdot \left(\#\ho^0(\X_{N},\li^{\otimes d})\right)=1
	\]
	is a result of the inequalities (\ref{2.10.1}) and (\ref{2.10.2}) by making $d$ tend to infinity.
\end{proof}
\begin{cor}\label{cor}
	For a fixed integer $N>1$, choose a subset $E_d\subset \mathrm{H}^0(\X_N, \li^{\otimes d})$ for each $d>0$ and write $E= \bigcup_{d>0} E_d$. Then $E$ is of density $\rho$ (defined as the limit $\lim_{d\rightarrow \infty}\frac{\#E_d}{\#\mathrm{H}^0(\X_N, \li^{\otimes d})}$ if exists) for some $0\leq \rho \leq 1$ if and only if we have
	\begin{eqnarray*}
		\lim_{d\rightarrow \infty} \frac{\sum_{\sigma \in \phi_{d,N}^{-1}(E_d)} (e^{-\pi ||\sigma||^2})}{\sum_{\sigma \in \mathrm{H}^0(\X, \Li^{\otimes d})}(e^{-\pi ||\sigma||^2})}=\rho
	\end{eqnarray*}
	i.e. if and only if the subset $E'= \bigcup_{d>0}  \phi_{d,N}^{-1}(E_d)$ is of $\theta$-density $\rho$.
\end{cor}
\begin{proof}
	As for any $\sigma'\in \mathrm{H}^0(\X_N, \li^{\otimes d})$ and any $\sigma'_0\in \mathrm{H}^0(\X_N, \li^{\otimes d})$ with image $\sigma'$ by $\phi_{d,N}$, we have
	\[
		\phi_{d,N}^{-1}(\sigma')=\sigma'_0+\mathrm{Ker}(\phi_{d,N}),
	\]
	the above proposition tells us that for any $\sigma'\in \mathrm{H}^0(\X_N, \li^{\otimes d})$,
	\begin{equation}
		\frac{\sum_{\sigma\in \phi_{d,N}^{-1}(\sigma')} \exp(-\pi \lVert\sigma\rVert^2) }{ \sum_{\sigma \in \mathrm{H}^0(\X, \Li^{\otimes d})} \exp(-\pi \lVert\sigma\rVert^2) }
		< \frac{\exp\left(h^1_{\theta}(\mathrm{Ker}(\phi_{d,N}))\right)}{\#\ho^0(\X_{N},\li^{\otimes d})},
	\end{equation}
	and
	\begin{equation}
		\frac{\sum_{\sigma\in\phi_{d,N}^{-1}(\sigma')} \exp(-\pi \lVert\sigma\rVert^2) }{ \sum_{\sigma \in \mathrm{H}^0(\X, \Li^{\otimes d})} \exp(-\pi \lVert\sigma\rVert^2) }
		>\frac{\left[2\cdot\exp\left(-h^1_{\theta}(\X, \Li^{\otimes d} )\right)- \exp\left(h^1_{\theta}(\mathrm{Ker} \phi_{d,N} )\right) \right]}{\#\ho^0(\X_{N},\li^{\otimes d})}. 
	\end{equation}
	Summing up over all $\sigma'\in E_d$, we get
	\[
		 \frac{\sum_{\sigma \in \phi_{d,N}^{-1}(E_d)} (e^{-\pi ||\sigma||^2})}{\sum_{\sigma \in \mathrm{H}^0(\X, \Li^{\otimes d})}(e^{-\pi ||\sigma||^2})}
		 < \exp\left(h^1_{\theta}(\mathrm{Ker}(\phi_{d,N}))\right) \frac{\#E_d}{\#\ho^0(\X_{N},\li^{\otimes d})},
	\]
	and
	\[
		 \frac{\sum_{\sigma \in \phi_{d,N}^{-1}(E_d)} (e^{-\pi ||\sigma||^2})}{\sum_{\sigma \in \mathrm{H}^0(\X, \Li^{\otimes d})}(e^{-\pi ||\sigma||^2})}
		 >\left[2\cdot\exp\left(-h^1_{\theta}(\X, \Li^{\otimes d} )\right)- \exp\left(h^1_{\theta}(\mathrm{Ker} \phi_{d,N} )\right) \right] \frac{\#E_d}{\#\ho^0(\X_{N},\li^{\otimes d})},
	\]
	As $N$ is fixed, we have
	\[
		\lim_{d\rightarrow \infty } \exp\left(h^1_{\theta}(\mathrm{Ker}(\phi_{d,N}))\right)=\lim_{d\rightarrow \infty } \left[2\cdot\exp\left(-h^1_{\theta}(\X, \Li^{\otimes d} )\right)- \exp\left(h^1_{\theta}(\mathrm{Ker} \phi_{d,N} )\right) \right]=1.
	\]
	Hence
	\[
		\lim_{d\rightarrow \infty} \frac{\sum_{\sigma \in \phi_{d,N}^{-1}(E_d)} (e^{-\pi ||\sigma||^2})}{\sum_{\sigma \in \mathrm{H}^0(\X, \Li^{\otimes d})}(e^{-\pi ||\sigma||^2})}=\lim_{d\rightarrow \infty} \frac{\#E_d}{\#\ho^0(\X_{N},\li^{\otimes d})}
	\]
	and we conclude.
\end{proof}

\begin{cor}\label{thetadensity}
	Let $\X$ be a projective arithmetic variety, and let $\Li$ be an ample Hermitian line bundle on $\X$. Let $\varepsilon_0$ be a constant as in Proposition \ref{epsilon}, and choose a constant $\delta<\varepsilon_0$. For an integer $N>1$, let 
	$\phi_{d,N}: \ho^0(\X,\Li^{\otimes d}) \longrightarrow \ho^0(\X_{N},\li^{\otimes d})$
	be the restriction map. When $d$ is large enough, we can find a constant $K_0>0$ such that for any $N\leq e^{\delta d}$ and any subset $E\subset \ho^0(\X_{N},\li^{\otimes d}) $,
	\begin{eqnarray*}
		\left| \frac{\sum_{\sigma\in \phi_{d,N}^{-1}(E)} \exp(-\pi \lVert\sigma\rVert^2) }{ \sum_{\sigma \in \mathrm{H}^0(\X, \Li^{\otimes d})} \exp(-\pi \lVert\sigma\rVert^2) } -  \frac{\#E}{\#\ho^0(\X_{N},\li^{\otimes d})}\right| < 10 \exp \big(K_0d^{n-1}-\pi e^{2(\varepsilon_0-\delta) d}\big)\cdot \frac{\#E}{\#\ho^0(\X_{N},\li^{\otimes d})} .
	\end{eqnarray*}
	In particular, when $d$ is sufficiently large, for any $N\leq e^{\delta d}$ we have
	\[
		\frac{\sum_{\sigma\in \phi_{d,N}^{-1}(E)} \exp(-\pi \lVert\sigma\rVert^2) }{ \sum_{\sigma \in \mathrm{H}^0(\X, \Li^{\otimes d})} \exp(-\pi \lVert\sigma\rVert^2) } \leq 2\frac{\#E}{\#\ho^0(\X_{N},\li^{\otimes d})}.
	\]
\end{cor}
\begin{proof}
	When $d$ is sufficiently large, $\ho^0(\X,\Li^{\otimes d})$ has a basis consisting of sections with norm smaller than $e^{-\varepsilon_0 d}$. Then Lemma \ref{minimum} tells us that any non-zero element of $\ho^0(\X,\Li^{\otimes d})^{\vee}$ has norm larger than $e^{\varepsilon_0 d}$. As $\Ker\phi_{d,N}\simeq N\cdot \ho^0(\X,\Li^{\otimes d})$, any non-zero element of $\Ker\phi_{d,N}^{\vee}$ has norm larger than $N^{-1}e^{\varepsilon_0 d}$. When $N\leq e^{\delta d}$, we have $N^{-1}e^{\varepsilon_0 d} \geq e^{(\varepsilon_0-\delta) d}$. We may apply Proposition 2.6.2 of \cite{Bo20} and get
	\begin{eqnarray*}
		0<\exp(h^1_{\theta}(\X,\Li^{\otimes d}))-1 &\leq& 3^{\mathrm{rk}\big(\ho^0(\X,\Li^{\otimes d})\big)}\left(1-\frac{\mathrm{rk}\big(\ho^0(\X,\Li^{\otimes d})\big)}{2\pi e^{2\varepsilon_0 d}} \right)^{-1}e^{-\pi e^{2\varepsilon_0 d}}, \\
		0<\exp(h^1_{\theta}(\Ker\phi_{d,N}))-1 &\leq& 3^{\mathrm{rk}\big(\ho^0(\X,\Li^{\otimes d})\big)}\left(1-\frac{\mathrm{rk}\big(\ho^0(\X,\Li^{\otimes d})\big)}{2\pi e^{2(\varepsilon_0-\delta) d}} \right)^{-1}e^{-\pi e^{2(\varepsilon_0-\delta) d}}.
	\end{eqnarray*}
	Note that when $d$ is sufficiently large, 
	\[
		\mathrm{rk}\big(\ho^0(\X,\Li^{\otimes d})\big)=\frac{\big((\li|_{\X_{\Q}})^{n-1} \big)}{(n-1)!} d^{n-1}+O(d^{n-2}),
	\]
	by asymptotic Riemann-Roch theorem. We can find a constant $K'_0>0$ such that 
	\[
		\mathrm{rk}\big(\ho^0(\X,\Li^{\otimes d})\big)\leq K'_0 d^{n-1}. 
	\]
	Set $K_0=K'_0\log 3$. Then when $d$ is large enough, for any $N\leq e^{\delta d}$ we have
	\begin{eqnarray*}
		\exp(h^1_{\theta}(\X,\Li^{\otimes d}))-1 &\leq&2 \exp\big( K_0d^{n-1}-\pi e^{2\varepsilon_0 d} \big)  \\
		\exp(h^1_{\theta}(\Ker\phi_{d,N}))-1&\leq& 2\exp \big(K_0d^{n-1}-\pi e^{2(\varepsilon_0-\delta) d}\big).
	\end{eqnarray*}
	In particular, we have $\exp(h^1_{\theta}(\X,\Li^{\otimes d}))-1\leq 1$, i.e. $\exp(h^1_{\theta}(\X,\Li^{\otimes d}))\leq 2$. Similarly we also have $\exp(h^1_{\theta}(\Ker\phi_{d,N}))\leq 2$.\\
	
	Now Proposition \ref{restheta} tells us that for any $\sigma_0\in \mathrm{H}^0(\X,\overline{\mathcal{L}}^{\otimes d})$,
	\begin{eqnarray*}
		& &\frac{\sum_{\sigma\in \mathrm{Ker}(\phi_{d,N})} \exp(-\pi \lVert\sigma+\sigma_0\rVert^2) }{ \sum_{\sigma \in \mathrm{H}^0(\X, \overline{\mathcal{L}}^{\otimes d})} \exp(-\pi \lVert\sigma\rVert^2) }
		-\frac{1}{\#\ho^0(\X_{N},\li^{\otimes d})} \\
		&<& \frac{\exp\left(h^1_{\theta}(\mathrm{Ker}(\phi_{d,N}))\right)-1}{\#\ho^0(\X_{N},\li^{\otimes d})}\\
		&\leq &2 \exp\big( K_0d^{n-1}-\pi e^{2\varepsilon_0 d} \big) \frac{1}{\#\ho^0(\X_{N},\li^{\otimes d})},
	\end{eqnarray*}
	and
	\begin{eqnarray*}
		& &\frac{\sum_{\sigma\in \mathrm{Ker}(\phi_{d,N})} \exp(-\pi \lVert\sigma+\sigma_0\rVert^2) }{ \sum_{\sigma \in \mathrm{H}^0(\X, \overline{\mathcal{L}}^{\otimes d})} \exp(-\pi \lVert\sigma\rVert^2) }
		-\frac{1}{\#\ho^0(\X_{N},\li^{\otimes d})}\\
		&>&\frac{2\cdot\exp\left(-h^1_{\theta}(\X, \Li^{\otimes d} )\right)- \exp\left(h^1_{\theta}(\mathrm{Ker} \phi_{d,N} )\right) -1}{\#\ho^0(\X_{N},\li^{\otimes d})}\\
		&=&\frac{2\left(\exp\left(-h^1_{\theta}(\X, \Li^{\otimes d} )\right)-1\right)- \left(\exp\left(h^1_{\theta}(\mathrm{Ker} \phi_{d,N} )\right) -1\right)}{\#\ho^0(\X_{N},\li^{\otimes d})}\\
		&\geq&-\frac{2\exp(h^1_{\theta}(\X,\Li^{\otimes d}))\cdot 2 \exp\big( K_0d^{n-1}-\pi e^{2\varepsilon_0 d} \big)+2\exp \big(K_0d^{n-1}-\pi e^{2(\varepsilon_0-\delta) d}\big)}{\#\ho^0(\X_{N},\li^{\otimes d})}\\
		&=&-\frac{8\exp\big( K_0d^{n-1}-\pi e^{2\varepsilon_0 d} \big)+2\exp \big(K_0d^{n-1}-\pi e^{2(\varepsilon_0-\delta) d}\big)}{\#\ho^0(\X_{N},\li^{\otimes d})}\\
		&\geq&-10\exp \big(K_0d^{n-1}-\pi e^{2(\varepsilon_0-\delta) d}\big)\frac{1}{\#\ho^0(\X_{N},\li^{\otimes d})}.
	\end{eqnarray*}
	
	In any case, we have
	\begin{eqnarray*}
		& &\left| \frac{\sum_{\sigma\in \mathrm{Ker}(\phi_{d,N})} \exp(-\pi \lVert\sigma+\sigma_0\rVert^2) }{ \sum_{\sigma \in \mathrm{H}^0(\X, \overline{\mathcal{L}}^{\otimes d})} \exp(-\pi \lVert\sigma\rVert^2) }
		-\frac{1}{\#\ho^0(\X_{N},\li^{\otimes d})} \right| \\
		&<& 10\exp \big(K_0d^{n-1}-\pi e^{2(\varepsilon_0-\delta) d}\big)\frac{1}{\#\ho^0(\X_{N},\li^{\otimes d})}.
	\end{eqnarray*}
	Therefore for any subset $E\subset \ho^0(\X_{N},\li^{\otimes d}) $, we have
	\begin{eqnarray*}
		& &\left| \frac{\sum_{\sigma\in \phi_{d,N}^{-1}(E)} \exp(-\pi \lVert\sigma\rVert^2) }{ \sum_{\sigma \in \mathrm{H}^0(X, \overline{\mathcal{L}}^{\otimes d})} \exp(-\pi \lVert\sigma\rVert^2) } -  \frac{\#E}{\#\ho^0(\X_{N},\li^{\otimes d})}\right| \\
		&\leq& \sum_{\gamma\in E}\left| \frac{\sum_{\sigma\in \phi_{d,N}^{-1}(\gamma)} \exp(-\pi \lVert\sigma+\sigma_0\rVert^2) }{ \sum_{\sigma \in \mathrm{H}^0(\X, \overline{\mathcal{L}}^{\otimes d})} \exp(-\pi \lVert\sigma\rVert^2) }
		-\frac{1}{\#\ho^0(\X_{N},\li^{\otimes d})} \right| \\
		&<& 10 \exp \big(K_0d^{n-1}-\pi e^{2(\varepsilon_0-\delta) d}\big)\cdot \frac{\#E}{\#\ho^0(\X_{N},\li^{\otimes d})},
	\end{eqnarray*}
	which finishes the proof.
\end{proof}

\section{Arithmetic Bertini theorem on irreducibility}\label{Irred}

In this section, we prove Theorem \ref{mainirred}. This is the $\theta$-density version of Theorem 1.6 in \cite{Ch21}. Our method follows the proof of Charles.

In the proof, we need to separate the two cases where the dimension of the projective arithmetic variety $\X$ is strictly larger than $2$ and where we have $\dim \X=2$. 

In the first case, we can apply the Bertini theorem on irreducibility over finite fields, proved by Charles and Poonen in \cite{CP16}, as the fibers over prime numbers $p$ of $\X$ are all of dimension larger than or equal to $2$. The dimension $2$ case needs special treatment as the theorem over finite fields cannot be applied. In particular, we need a resolution of singularity, and the arithmetic intersection theory for effective computations.

\subsection{Case of large dimension}
First, we prove Theorem \ref{mainirred} when the projective arithmetic variety $\X$ is of dimension $n>2$.

\begin{lem}\label{ver}
	Let $\X$ be a projective arithmetic variety of dimension $n>2$, and let $\Li$ be an ample Hermitian line bundle on $\X$. Then the set
	\[
		\left\{ \sigma\in \bigcup_{d>0} \mathrm{H}^0(\X, \Li^{\otimes d})\ ; \ \mathrm{div}(\sigma)\  \text{does not have any vertical component} \right\}
	\]
	is of $\theta$-density $1$.
\end{lem}

\begin{proof}
	Write
	\[
		\mathcal{R}_d^{ver}=\{ \sigma\in \mathrm{H}^0(\X, \Li^{\otimes d}) \ ;\ \mathrm{div}(\sigma) \text{ has a vertical component} \}.
	\]
	It suffices to show that the set $\mathcal{R}^{ver}:=\bigcup_{d>0}\mathcal{R}_d^{ver}$ is of $\theta$-density $0$. Choose a constant $\varepsilon>0$ sufficiently small and an arithmetic curve $C$ in $\X$. 
	Set
	\begin{eqnarray*}
		\mathcal{R}_d^{ver,1} &=& \{ \sigma\in \mathrm{H}^0(\X, \Li^{\otimes d}) \ ;\ \mathrm{div}(\sigma) \text{ has a vertical component in } \X_p \text{ with } p<\exp(\varepsilon d^2) \} \\
		\mathcal{R}_d^{ver,2} &=& \{ \sigma\in \mathrm{H}^0(\X, \Li^{\otimes d}) \ ;\ \mathrm{div}(\sigma) \supset C \}  \\
		\mathcal{R}_d^{ver,3} &=& \{ \sigma\in \mathrm{H}^0(\X, \Li^{\otimes d}) \ ;\ h_{\overline{\mathcal L}}(\mathrm{div}(\sigma) . C)\geq \varepsilon d^2 \}.
	\end{eqnarray*}
	Here $h_{\Li}$ is the height function on $\X$ defined by the ample Hermitian line bundle $\Li$. In particular, if $\mathrm{div}(\sigma). C=\sum_{i}n_i x_i$ is a cycle of dimension $0$, then
	\[
		h_{\Li}(\mathrm{div}(\sigma) . C)=\sum_{i}n_i \log (\#\kappa(x_i)).
	\]
	If $\sigma\in \mathrm{H}^0(\X, \Li^{\otimes d})$ is such that $\mathrm{div}(\sigma)$ has a vertical component in $\X_p$ with $p\geq \exp(\varepsilon d^2)$ and that $C$ is not contained in $\mathrm{div}(\sigma)$, then $\mathrm{div}(\sigma) \cap C$ contains at least a closed point of $\X_p$, hence
	\[
		h_{\Li}(\mathrm{div}(\sigma) . C)\geq \log p\geq \varepsilon d^2.
	\]
	Therefore we have
	\[
		\mathcal{R}_d^{ver}\subset \mathcal{R}_d^{ver,1}\cup \mathcal{R}_d^{ver,2}\cup \mathcal{R}_d^{ver,3}.
	\]
	\medskip
	
	Note that by Proposition \ref{fibre}, we may choose $d_0$ and a constant $\eta>0$ so that when $d\geq d_0$, for any prime number $p$,  
	\[
		\frac{\sum_{\sigma \in \Ker \phi_{d,p}} (e^{-\pi ||\sigma||^2})}{\sum_{\sigma\in \mathrm{H}^0(\X, \Li^{\otimes d}) } \exp{(-\pi||\sigma||^{2})}}=O\left(\exp{(-\eta d^{n-1})}\right).
	\]
	So the sections in $\mathcal{R}_d^{ver,1}$ can be controled by 
	\begin{eqnarray*}
		& &\frac{\sum_{\sigma \in \mathcal{R}_d^{ver,1}} (e^{-\pi ||\sigma||^2})}{\sum_{\sigma\in \mathrm{H}^0(\X, \Li^{\otimes d}) } \exp{(-\pi||\sigma||^{2})}}\\
		&\leq&\sum_{p<\exp(\varepsilon d^2)}\frac{\sum_{\sigma \in \Ker \phi_{d,p}} (e^{-\pi ||\sigma||^2})}{\sum_{\sigma\in \mathrm{H}^0(\X, \Li^{\otimes d}) } \exp{(-\pi||\sigma||^{2})}}\\
		&=&O\left( \exp(\varepsilon d^2-\eta d^{n-1}) \right) =o(1).
	\end{eqnarray*}
	Moreover, $\mathcal{R}_d^{ver,2}$ is exactly the set of sections in $\mathrm{H}^0(\X, \Li^{\otimes d})$ which vanishes identically on $C$, i.e. $\mathcal{R}_d^{ver,2}=\mathrm{Ker}(\phi_{d, C})$, where $\phi_{d, C}$ is the restriction to $C$ morphism. Again by Proposition \ref{fibre} we have
	\begin{eqnarray*}
			\frac{\sum_{\sigma\in \mathcal{R}_d^{ver,2}} \exp{(-\pi||\sigma||^{2})}}{\sum_{\sigma\in \mathrm{H}^0(\X, \Li^{\otimes d}) } \exp{(-\pi||\sigma||^{2})}} = O\left(\exp(-\eta \cdot d )\right)=o(1).
	\end{eqnarray*}
	So it remains to controle sections in $\mathcal{R}_d^{ver,3}$.

	Note that the height of $\mathrm{div}(\sigma).C$ can be calculated via the formula
	\[
		h_{\Li}(\mathrm{div}(\sigma) \cdot C)=d\cdot h_{\overline{\mathcal L}}(C)+ \deg [K_C:\mathbb{Q}]\cdot \log ||\sigma||,
	\]
	where we denote by $K_C$ the field of definition of the arithmetic curve $C$. So if $\sigma \in \mathcal{R}_d^{ver,3}$, we have
	\[
		||\sigma||\geq \exp\left( \frac{\varepsilon d^2- d\cdot h_{\Li}(C)}{\deg [K_C:\mathbb{Q}]} \right)
	\]
	By Lemme 3.2.1 in \cite{Bo20}, for any $t\in (0,1]$,
	\begin{eqnarray*}
		& & \left(\sum_{||\sigma|| \geq \exp\left( \frac{\varepsilon d^2- d\cdot h_{\Li}(C)}{\deg [K_C:\mathbb{Q}]} \right)}e^{-\pi ||\sigma||^2}\right) \cdot \exp\left(-\mathrm{h}^0_{\theta}(\X, \Li^{\otimes d})\right) \\
		&\leq& t^{-\mathrm{h}^0(\X_{\mathbb{Q}} , \Li^{\otimes d} )/2}\cdot e^{ -\pi(1-t)\exp\left( \frac{2\varepsilon d^2- 2d\cdot h_{\Li}(C)}{\deg [K_C:\mathbb{Q}]} \right) } \\
		&=& o(1).
	\end{eqnarray*}
	By putting these three estimates together, we get the result.
\end{proof}
	
\begin{proof}[Proof of Theorem \ref{mainirred} (case when $dim \X>2$)]
	Let $p$ be a sufficiently large prime number such that the fiber $\X_p$ is reduced. The specialisation gives a bijection between the set of irreducible components of $\X_{\overline{\mathbb Q}}$ and the set of irreducible components of $\X_{\overline{\mathbb F}_p}$. Let $\X_{0,p}$ be an irreducible component of $\X_p$ equipped with the induced reduced structure. Take a section $\sigma\in \mathrm{H}^0(\X, \Li^{\otimes d})$. If $D$ is a horizontal irreducible component of $\mathrm{div}(\sigma)$, then $D$ intersects every irreducible component of $\X_{\overline{\mathbb Q}}$, and hence intersects $\X_{0,p}$. So if $\mathrm{div}(\sigma|_{\X_{0,p}})$ is irreducible and that $\mathrm{div}(\sigma)$ has no vertical component, then $\mathrm{div}(\sigma)$ has only one horizontal irreducible component, which suggests that it is irreducible. 
	
	By Theorem 1.6 of \cite{CP16}, the set of sections $\sigma_p\in \bigcup_{d>0}\mathrm{H}^0(\X_{0,p}, \li^{\otimes d})$ such that $\mathrm{div}(\sigma_p)$ is irreducible is of density $1$. Applying Corollary \ref{cor}, we get that the set of sections $\sigma\in \bigcup_{d>0}\mathrm{H}^0(\X,\Li^{\otimes d})$ such that $\mathrm{div}(\sigma|_{\X_{0,p}})$ is irreducible is of $\theta$-density $1$.
		
	On the other hand, Lemma \ref{ver} tells us that the set of sections $\sigma\in \bigcup_{d>0}\mathrm{H}^0(\X,\Li^{\otimes d})$ such that $\mathrm{div}(\sigma)$ does not have a vertical component is of $\theta$-density $1$. We conclude by combining these two result.
\end{proof}

\subsection{Case of arithmetic surface}
In this part we follow the proof of Charles in Section 5 of \cite{Ch21}.  
	
We choose at first real numbers $0<\beta<\alpha<\frac{1}{2}$. For any $d>0$, let $Y_d\subset \mathrm{H}^0(\X, \Li^{\otimes d})$ be the subset of sections $\sigma$ such that :
\begin{enumerate}
	\item $\sigma$ does not vanish on a Weil divisor $D$ of $\X$ with $h_{\Li}(D)\leq d^{\alpha}$;\\
	\item there exists an irreducible component $D'$ of $\mathrm{div}(\sigma)$ with
	\[
		\deg (D'_{\mathbb Q})\geq (\deg \Li_{\mathbb Q})\cdot d- rd^{\beta}.
	\]
\end{enumerate}
We first show that it suffices to consider sections in $Y:=\bigcup_{d>0}Y_d$.
\begin{prop}
	The subset $Y\subset \bigcup_{d>0}\mathrm{H}^0(\X, \Li)^{\otimes d}$ is of $\theta$-density $1$.
\end{prop}
\begin{proof}
	Assume that $d$ is large enough. By Theorem B of \cite{Mo04}, there exists a constant $C$ such that the number of Weil divisors $D$ with $h_{\Li}(D)\leq d^{\alpha}$ is bounded above by $e^{Cd^{2\alpha}}$. By Proposition \ref{fibre}, there exists constants $C',\eta$ such that for any divisor $D$ as above, we have 
	\[
		\frac{\sum_{\sigma\in \mathrm{Ker}(\phi_{d, D}) } \exp{(-\pi||\sigma||^{2})}}{\sum_{\sigma\in \mathrm{H}^0(\X, \Li^{\otimes d}) } \exp{(-\pi||\sigma||^{2})}}
		\leq C'\exp{(-\eta d)}.
	\]
	Therefore the $\theta$-proportion of sections vanishing on a Weil divisor of height smaller than or equal to $d^{\alpha}$ is bounded above by $C'\exp(Cd^{2\alpha}-\eta d) $, which tends to $0$ when $d$ tends to infinity.

	Let $\tau$ be a real number with $1-\beta< \tau<1$, $t_d=\lceil d^{\tau} \rceil$, and let $\mathcal{S}$ be a non-empty open subscheme of $\mathrm{Spec}\  \mathbb{Z}$ such that the restriction $\X_{\mathcal{S}}\rightarrow \mathcal{S}$ has reduced fibers. Let $p_1, \cdots, p_{t_d}$ be the first $t_d$ smallest prime numbers in $\mathcal{S}$ and set $N_d=\prod_{i=1}^{t_d}p_i$. Then when $d$ is sufficiently large, by Lemma \ref{freemod} we have
	\[
		\mathrm{H}^0(\X_{N_d},\li^{\otimes d})=\mathrm{H}^0(\X,\li^{\otimes d})/\big(N_d\cdot \mathrm{H}^0(\X,\li^{\otimes d})\big);
	\]
	Moreover, the proof of Proposition 5.6 in \cite{Ch21} gives us the following two conclusions.
	\begin{itemize}
		\item Denote by $E_p$ the subset of sections $s\in \mathrm{H}^0(\X_p,\li^{\otimes d})$ such that there exists an irreducible component $C$ of $\X_p$ having $r_C$ geometric irreducible components, that $s|_{C}\not=0$ and that $s$ vanishes on an irreducible divisor $D_p$ of $\X_p$ of degree at least $r_C(2d-d^{\beta})$. There exists positive real numbers $A,B$ only depending on $\beta$ and $\X_{\mathcal{S}}\rightarrow \mathcal{S}$ such that if $d>A$, the proportion of sections $\sigma\in \mathrm{H}^0(\X_{N_d}, \li^{\otimes d})$ which are not projected on any of the $E_{p_i}$ for $1\leq i\leq t_d$ is bounded above by
		\[
		 	(1-Bd^{\beta-1})^{t_d}\leq (1-Bd^{\beta-1})^{d^{\tau}+1}=\exp(-Bd^{\beta+\tau+1}+o(d^{\beta+\tau+1}))=o(1).
		\]
		Therefore the proportion of $\sigma\in \mathrm{H}^0(\X_{N_d},\li^{\otimes d})$ which are not projected to at least one of the $E_{p_i}$ with $1\leq i\leq t_d$ tends to $1$ when $d\rightarrow \infty$.
		\item If $\sigma\in \mathrm{H}^0(\X,\Li^{\otimes d})$ is such that $\sigma_p\in E_p$ for some $p\in \mathcal{S}$, (we denote by $D_p$ the corresponding irreducible divisor), then there exists an irreducible component $D$ of $\mathrm{div}(\sigma)$ with $D\cap \X_p=D_p$ and
		\[
			\deg (D_{\mathbb Q}) \geq (\deg \Li_{\mathbb Q})\cdot d- rd^{\beta}.
		\]
		In other words, if $\sigma$ is contained in $Y_d$, then for some $p\in \mathcal{S}$ it should be projected to $E_p$.
	\end{itemize}
	To finish the proof, we apply Corollary \ref{thetadensity}. The $\theta$-proportion of $\sigma\in \mathrm{H}^0(\X, \Li^{\otimes d})$ which are not projected to any one of the $E_{p_i}$ for $1\leq i\leq t_d$ is bounded above by
	\begin{eqnarray*}
		& &\exp(-Bd^{\beta+\tau+1}+o(d^{\beta+\tau+1}))\cdot \exp(h_{\theta}^1(\mathrm{Ker}(\phi_{d, N_d}))) \\
		&=& \exp(h_{\theta}^1(\mathrm{Ker}(\phi_{d, N_d}))-Bd^{\beta+\tau+1}+o(d^{\beta+\tau+1})).
	\end{eqnarray*}
	Moreover, Proposition 2.6.2 in \cite{Bo20} shows that as $\mathrm{Ker}(\phi_{d, N_d})^{\vee}$ has a basis of norm at least $\frac{1}{N_d}e^{\varepsilon_0 d}$,
	\[
		h_{\theta}^1(\mathrm{Ker}(\phi_{d, N_d}))<2\cdot 3^{h^0(\X_{\mathbb Q}, \Li_{\mathbb Q}^{\otimes d})}\exp(-\frac{\pi}{N_d^2}e^{2\varepsilon_0 d}).
	\]
	Since $N_d=\prod_{i=1}^{t_d}p_i$ and that $p_i\sim i\log i$, we have 
	\[
		N_d=O\left( (t_d\log t_d)^{t_d} \right)=O\left( (d^{\tau}\tau\log d)^{d^{\tau}} \right)=O\left( d^{d^{\tau}} \right)=O\left( e^{d^{\tau}\log d} \right).
	\]
	as $\tau<1$. Since $\X_{\mathbb Q}$ is of dimension $1$, we have $h^0(\X_{\mathbb Q}, \Li_{\mathbb Q}^{\otimes d})=O(d)$. We may choose a constant $C_1>0$ so that $h^0(\X_{\mathbb Q}, \Li_{\mathbb Q}^{\otimes d})\leq C_1 d$ and $N_d\leq C_1e^{d^{\tau}\log d}$. Then we may have  
	\begin{eqnarray*}
		h_{\theta}^1(\mathrm{Ker}(\phi_{d, N_d}))\leq 2\exp\big( C_1d\log 3-\frac{\pi}{C_1^2} e^{2\varepsilon_0 d-2d^{\tau}\log d}\big)
	\end{eqnarray*}
	Therefore we obtain 
	\begin{eqnarray*}
		& & \exp(h_{\theta}^1(\mathrm{Ker}(\phi_{d, N_d}))-Bd^{\beta+\tau+1}+o(d^{\beta+\tau+1})) \\
		&<& \exp(2\cdot 3^{h^0(\X_{\mathbb Q}, \Li_{\mathbb Q}^{\otimes d})}\exp(-\frac{\pi}{N_d^2}e^{2\varepsilon_0 d})-Bd^{\beta+\tau+1}+o(d^{\beta+\tau+1})) \\
		&<& \exp\left(2\exp\big( C_1d\log 3-\frac{\pi}{C_1^2} e^{2\varepsilon_0 d-2d^{\tau}\log d}\big)-Bd^{\beta+\tau+1}+o(d^{\beta+\tau+1})\right) \\
		&=& o(1).
	\end{eqnarray*}
	So the $\theta$-density of sections $\sigma\in \mathrm{H}^0(\X, \Li^{\otimes d})$ which are not sent to any of the $E_{p_i}$ for $1\leq i\leq t_d$ is equal to $0$. This concludes our proof.
\end{proof}
	
Denote by $Z_d\subset Y_d$ the subset of sections $\sigma$ such that $\mathrm{div}(\sigma)$ is irreducible. To prove Theorem \ref{mainirred}, it suffices to show that $Z=\bigcup_{d>0}Z_d$ is of $\theta$-density $0$.

Let $\pi : \widetilde{\X}\longrightarrow \X$ be a resolution of singularity of $\X$, $r$ be the number of irreducible components of $\widetilde{\X}$. Set $\overline{\mathcal B}=\pi^*\Li$. Let $\widehat{\mathrm Pic}_{\omega}(\widetilde \X)$ be the group of isomorphism classes of $\omega$-admissible Hermitian line bundles on $\widetilde \X$ with $\omega$ the first Chern class of $\overline{\mathcal B}$. So we have the short exact sequence
\[
	0\longrightarrow \mathbb{R} \longrightarrow \widehat{\mathrm Pic}_{\omega}(\widetilde \X) \longrightarrow {\mathrm Pic}(\widetilde \X) \longrightarrow 0.
\]

We fix a subgroup $N$ of $\widehat{\mathrm Pic}_{\omega}(\widetilde \X)$ satisfying the following conditions :
\begin{enumerate}[$\bullet$]
	\item $N$ is a group of finite type ;
	\item $N$ contains $\overline{\mathcal B}$ ;
	\item the induced morphism $N\longrightarrow\mathrm{Pic}(\widetilde{\X})$ is surjective ;
	\item $N\cap \mathrm{Ker}(\widehat{\mathrm Pic}_{\omega}(\widetilde \X) \longrightarrow {\mathrm Pic}(\widetilde \X))$ is of rank $1$.
\end{enumerate}
By Lemma 5.7 in \cite{Ch21}, we can find a positive integer $k$ and a decompostion $\overline{\mathcal B}^{k}\simeq \overline{\mathcal A}\otimes \overline{\mathcal E}$ to an ample Hermitian line bundle $\overline{\mathcal A}$ and an effective line bundle $\overline{\mathcal E}$ (i.e. containing a global section of norm $< 1$).

Choose real numbers $\delta, \gamma$ such that $0<\beta<\gamma<\delta<\alpha< \frac{1}{2}$. The following lemma is a variant of Lemma 5.9 in \cite{Ch21}. 
\begin{lem}\label{decomp}
	When $d\gg 1$, for any $\sigma\in Z_d$, we can associate with it two Hermitian line bundles $\overline{\mathcal L}_1, \overline{\mathcal L}_2$ on $\widetilde{\X}$ and two global sections $\sigma_1\in \mathrm{H}^0(\widetilde{\X}, \Li_1),\ \sigma_2\in \mathrm{H}^0(\widetilde{\X}, \Li_2)$ satisfying the following conditions :
	\begin{enumerate}[(i)]
		\item $\Li_1, \Li_2 \in N$ ;
		\item $d^{\delta}\leq \Li_1. \overline{\mathcal B}\leq d \overline{\mathcal B}. \overline{\mathcal B}-d^{\delta}$ ; 	
		\item $\Li_1 . \overline{\mathcal A}\leq d \overline{\mathcal B}. \overline{\mathcal B}$ ;
		\item $\Li_1\otimes \Li_2\simeq \overline{\mathcal B}^{d}$ ;
		\item $\pi^{*}(\sigma)=\sigma_1\cdot \sigma_2$ up to an isomorphism ;
		\item there exist constants $C_1,C_2,C_3$ such that 
		\begin{enumerate}[$\bullet$]
			\item $C_1^{-1}< ||\sigma_2||_0 < C_1$,
			\item $||\sigma_1||\cdot||\sigma_2||_0<(dC_2\deg(\mathcal{L}_{\mathbb Q}))^{rd^{\beta}}||\sigma||$, 
			\item $||\sigma_2||<C_3^{rd^{\beta}}||\sigma_2||_0$, $C_3>1$,	
			where $||\cdot ||_0$ is the norm of sections defined by 
			\begin{eqnarray*}
				||\sigma ||_0=\exp\Big( \int_{\widetilde{\X}(\bbC)}\log\lVert\sigma\rVert^2\omega \Big)
			\end{eqnarray*}
			for any global section $\sigma$ of a Hermitian line bundle $\Li$.
		\end{enumerate}
	\end{enumerate}
\end{lem}
Here the conditions on norms of sections is weakened. We can prove this lemma with exactly the same method as Lemma 5.9 in \cite{Ch21}.

Now we can finish the proof of the $\theta$-version of arithmetic Bertini theorem on irreducibility.
\begin{proof}[Proof of Theorem \ref{mainirred} (when $dim \X=2$)]
	Let $\sigma$ be a global section in $Z_d$. We can associate with it two Hermitian line bundles $\Li_1, \Li_2$ on $\widetilde{\X}$ and global sections $\sigma_1\in \mathrm{H}^0(\X, \Li_1),\ \sigma_2\in \mathrm{H}^0(\X, \Li_2)$ as in Lemma \ref{decomp}. Then
	\[
		||\sigma||=||\pi^{*}(\sigma)||\leq ||\sigma_1||\cdot||\sigma_2||.
	\]
	The condition (vi) of Lemma \ref{decomp} also gives a lower bound of the norm $||\sigma||$ by
	\begin{eqnarray*}
		||\sigma|| &\geq& (C_2\deg(\li_{\mathbb Q})\cdot d)^{-rd^{\beta}}\cdot C_3^{-rd^{\beta}}||\sigma_1||\cdot||\sigma_2|| \\
		&=&  (C_2C_3\deg(\li_{\mathbb Q})\cdot d)^{-rd^{\beta}}||\sigma_1||\cdot||\sigma_2||.
	\end{eqnarray*}
	So we have
	\begin{eqnarray*}
		& &\sum_{\sigma\in Z_d}e^{-\pi ||\sigma||^2} \\ 
		&\leq& \sum_{(\Li_1, \Li_2)}\left[ \sum_{C_1^{-1}<||\sigma_2||_0<C_1} \sum_{\sigma_1}\exp\left( -\pi(C_2C_3\deg(\mathcal{L}_{\mathbb Q})\cdot d)^{-2rd^{\beta}} ||\sigma_1||^2||\sigma_2||^2 \right)\right] \\
		&\leq& \sum_{(\Li_1, \Li_2)}\left[ \sum_{C_1^{-1}<||\sigma_2||<C_1C_3^{rd^{\beta}}} 
		\exp\left(h_{\theta}^0(\widetilde{\X}, \Li_1(rd^{\beta}\log d+rd^{\beta}\log(C_2C_3)\deg(\li_{\mathbb Q})-\log||\sigma_2||)) \right)\right] 
	\end{eqnarray*}
	Write $P(d,\sigma_2)=rd^{\beta}\log d+rd^{\beta}\log(C_2C_3)\deg(\mathcal{L}_{\mathbb Q})-\log||\sigma_2||$. The above inequality can be simplified to
	\begin{eqnarray*}
		\sum_{\sigma\in Z_d}e^{-\pi ||\sigma||^2}\leq   \sum_{(\Li_1, \Li_2)}\left[ \sum_{C_1^{-1}<||\sigma_2||<C_1C_3^{rd^{\beta}}} 
		\exp\left(h_{\theta}^0(\widetilde{\X}, \Li_1(P(d,\sigma_2))) \right)\right] 
	\end{eqnarray*}
	
	We write $l_1=\Li_1 . \overline{\mathcal B}$. We know from Lemma \ref{decomp} that $d^{\delta}\leq l_1\leq d \overline{\mathcal B} . \overline{\mathcal B}-d^{\delta}$. Therefore we have
	\begin{eqnarray*}
		\Li_1(P(d,\sigma_2)). \overline{\mathcal B} = l_1+P(d,\sigma_2)\overline{\mathcal{O}_{\widetilde{\X}}} . \overline{\mathcal B},
	\end{eqnarray*}
	where $\overline{\mathcal{O}_{\widetilde{\X}}}$ is the structure sheaf equipped with the trivial metric. Then applying Proposition 4.7 in \cite{Ch21}, we obtain that
	\begin{eqnarray*}
		& &h_{\theta}^0(\widetilde{\X}, \Li_1(P(d,\sigma_2))) \\
		&\leq & \frac{(\Li_1(P(d,\sigma_2)). \overline{\mathcal B})^2 }{2\overline{\mathcal B}. \overline{\mathcal B} } +O\left( (\Li_1(P(d,\sigma_2)). \overline{\mathcal B})\log( \deg\li_{1, \mathbb{Q}}) \right) \\
		& & +O(\deg\li_{1, \mathbb{Q}}\log(\deg\li_{1, \mathbb{Q}}))+O(1)  \\ \\
		&=& \frac{(l_1+P(d,\sigma_2)\overline{\mathcal{O}_{\widetilde{\X}}} . \overline{\mathcal B} )^2}{2\overline{\mathcal B}. \overline{\mathcal B} }+ O\left( (l_1+P(d,\sigma_2)\overline{\mathcal{O}_{\widetilde{\X}}}. \overline{\mathcal B} )\log(\deg\li_{1, \mathbb{Q}}) \right) \\
		& & +O\left( \deg\li_{1, \mathbb{Q}}\log( \deg\li_{1, \mathbb{Q}}) \right)+O(1) \\ \\
		&=& \frac{(l_1+P(d,\sigma_2)\overline{\mathcal{O}_{\widetilde{\X}}} . \overline{\mathcal B} )^2}{2\overline{\mathcal B}. \overline{\mathcal B} }+ O\left( (l_1+P(d,\sigma_2)\overline{\mathcal{O}_{\widetilde{\X}}}. \overline{\mathcal B} )\log(d ) \right) +O\left( d\log(d) \right).
	\end{eqnarray*}
	When $d\gg 1$, $P(d,\sigma_2)$ is positive for any $\sigma_2$ satisfying $C_1^{-1}\leq ||\sigma_2||\leq C_1C_3^{rd^{\beta}}$, as
	\[
		P(d, \sigma_2)\geq rd^{\beta}(\log d+\log(C_2C_3\deg \li_{\mathbb Q})-\log C_3)-\log C_1
	\]
	Moreover, since $\beta<\frac{1}{2}$,
	\[
		P(d, \sigma_2)\leq  rd^{\beta}(\log d+\log(C_2C_3\deg \li_{\mathbb Q}))+\log C_1=o(d).
	\]
	Finally, we obtain the following estimate :
	\[
		h_{\theta}^0(\widetilde{\X}, \Li_1(P(d,\sigma_2)))\leq \frac{(l_1+P(d,\sigma_2)\overline{\mathcal{O}_{\widetilde{\X}}} . \overline{\mathcal B} )^2}{2\overline{\mathcal B}. \overline{\mathcal B} }+O(d\log d).
	\]
	If we fix $\Li_1$ (then so is $\Li_2$ by Lemma \ref{decomp}), then we get
	\begin{eqnarray*}
		& &\sum_{C_1^{-1}<||\sigma_2||<C_1C_3^{rd^{\beta}}} \exp\left(h_{\theta}^0(\widetilde{\X}, \Li_1(P(d,\sigma_2))) \right) \\
		&\leq & \sum_{C_1^{-1}<||\sigma_2||<C_1C_3^{rd^{\beta}}}\exp\left( \frac{(l_1+P(d,\sigma_2)\overline{\mathcal{O}_{\widetilde{\X}}} . \overline{\mathcal B} )^2}{2\overline{\mathcal B}. \overline{\mathcal B} }+O(d\log d) \right) \\
		&\leq & \sum_{C_1^{-1}<||\sigma_2||<C_1C_3^{rd^{\beta}}}\exp\left( \frac{[ l_1+(rd^{\beta}(\log d+\log(C_2C_3\deg \li_{\mathbb Q}))+\log C_1 ) \overline{\mathcal{O}_{\widetilde{X}}}. \overline{\mathcal B}]^2}{2\overline{\mathcal B}. \overline{\mathcal B} }+O(d\log d) \right) \\ \\
		&\leq&  \Big(\#\mathrm{H}_{\mathrm{Ar}}^0(\widetilde{\X}, \Li_2(\log C_1+rd^{\beta}\log C_3))\Big)\\
		& & \qquad \cdot\exp\left( \frac{[ l_1+(rd^{\beta}(\log d+\log(C_2C_3\deg \li_{\mathbb Q}))+\log C_1 ) \overline{\mathcal{O}_{\widetilde{\X}}}. \overline{\mathcal B}]^2}{2\overline{\mathcal B}. \overline{\mathcal B} }+O(d\log d) \right).
	\end{eqnarray*}
	Then the arithmetic Hilbert-Samuel formula for $h^0_{\mathrm{Ar}}$ (see for example Theorem 2.2 in \cite{Wa22} ) tells us that
	\begin{eqnarray*}
		& & h_{\mathrm{Ar}}^0(\widetilde{\X}, \Li_2(\log C_1+rd^{\beta}\log C_3)) \\ \\
		&\leq& \frac{ \left(\Li_2 . \overline{\mathcal B} +(\log C_1+rd^{\beta}\log C_3)\overline{\mathcal{O}_{\widetilde{\X}}} . \overline{\mathcal B}\right)^2 }{2\overline{\mathcal B}. \overline{\mathcal B}}+O(d\log d) \\
		&=& \frac{ \left(d\overline{\mathcal B} . \overline{\mathcal B}-l_1 +(\log C_1+rd^{\beta}\log C_3)\overline{\mathcal{O}_{\widetilde{\X}}} . \overline{\mathcal B}\right)^2 }{2\overline{\mathcal B}. \overline{\mathcal B}}+O(d\log d).
	\end{eqnarray*}
	We also have
	\begin{eqnarray*}
		& & \log \left(\sum_{C_1^{-1}<||\sigma_2||<C_1C_3^{rd^{\beta}}} \exp\left(h_{\theta}^0(\widetilde{\X}, \Li_1(P(d,\sigma_2))) \right)   \right)\\ \\
		&\leq&  \frac{[ l_1+(rd^{\beta}(\log d+\log(C_2C_3\deg \li_{\mathbb Q}))+\log C_1 ) \overline{\mathcal{O}_{\widetilde{\X}}}. \overline{\mathcal B}]^2}{2\overline{\mathcal B}. \overline{\mathcal B} } \\
		& & \ +\frac{ \left(d\overline{\mathcal B} . \overline{\mathcal B}-l_1 +(\log C_1+rd^{\beta}\log C_3)\overline{\mathcal{O}_{\widetilde{\X}}} . \overline{\mathcal B}\right)^2 }{2\overline{\mathcal B}. \overline{\mathcal B}}+O(d\log d).
	\end{eqnarray*}
	To simplify, we can find a constant $K>0$ such that
	\begin{eqnarray*}
		Kd^{\beta} &>& (rd^{\beta}(\log d+\log(C_2C_3\deg \li_{\mathbb Q}))+\log C_1 ) \overline{\mathcal{O}_{\widetilde{\X}}}. \overline{\mathcal B}, \\
		Kd^{\beta} &>& (\log C_1+rd^{\beta}\log C_3)\overline{\mathcal{O}_{\widetilde{\X}}} . \overline{\mathcal B}.
	\end{eqnarray*}
	Therefore we have
	\begin{eqnarray*}
		& & \log \left(\sum_{C_1^{-1}<||\sigma_2||<C_1C_3^{rd^{\beta}}} \exp\left(h_{\theta}^0(\widetilde{\X}, \Li_1(P(d,\sigma_2))) \right)   \right) \\
		&<&  \frac{( l_1+Kd^{\beta})^2}{2\overline{\mathcal B}. \overline{\mathcal B} } +\frac{ \left(d\overline{\mathcal B} . \overline{\mathcal B}-l_1 +Kd^{\beta}\right)^2 }{2\overline{\mathcal B}. \overline{\mathcal B}}+O(d\log d) \\ \\
		&=& \frac{l_1^2+( d\overline{\mathcal B} . \overline{\mathcal B}-l_1)^2+ 2(l_1+ d\overline{\mathcal B} . \overline{\mathcal B}-l_1)\cdot Kd^{\beta}     +2Kd^{2\beta}}{2\overline{\mathcal B}. \overline{\mathcal B}} +O(d\log d) \\
		&=&  \frac{l_1^2+( d\overline{\mathcal B} . \overline{\mathcal B}-l_1)^2+ 2Kd^{1+\beta} +O(d)}{2\overline{\mathcal B}. \overline{\mathcal B}} +O(d\log d) \\
		&=& d^2\overline{\mathcal B} . \overline{\mathcal B}-2dl_1 +\frac{K}{\overline{\mathcal B} . \overline{\mathcal B}}d^{1+\beta} +O(d\log d) \\
		&\leq& d^2\overline{\mathcal B} . \overline{\mathcal B}-2d^{1+\delta}+\frac{K}{\overline{\mathcal B} . \overline{\mathcal B}}d^{1+\beta} +O(d\log d).
	\end{eqnarray*}
	This estimate does not depend on the choice of $\Li_1$. Moreover, in the proof of Proposition 5.10 in \cite{Ch21}, Charles also computed the number of $\overline{\mathcal L}_1$, which is bounded by $O(d^{\rho})$ with $\rho$ the rank of the subgroup $N$. 
		
	Finally, we get the following bound for sections in $Z_d$ :
	\begin{eqnarray*}
		& &\sum_{\sigma\in Z_d}e^{-\pi ||\sigma||^2} \\ 
		&\leq& \sum_{(\Li_1, \Li_2)}\left[ \sum_{C_1^{-1}<||\sigma_2||<C_1C_3^{rd^{\beta}}} 
		\exp\left(h_{\theta}^0(\widetilde{\X}, \Li_1(P(d,\sigma_2))) \right)\right] \\
		&\leq& O(d^{\rho})\cdot \exp\left[ d^2\overline{\mathcal B} . \overline{\mathcal B}-2d^{1+\delta}+\frac{K}{\overline{\mathcal B} . \overline{\mathcal B}}d^{1+\beta} +O(d\log d) \right] \\
		&=& \exp\left[ d^2\overline{\mathcal B} . \overline{\mathcal B}-2d^{1+\delta}+\frac{K}{\overline{\mathcal B} . \overline{\mathcal B}}d^{1+\beta} +O(d\log d) \right]
	\end{eqnarray*}
	On the other hand, we have
	\begin{eqnarray*}
		h_{\theta}^0(\X, \Li^{\otimes d})&=& h_{\theta}^0(\widetilde{\X}, \overline{\mathcal B}^{\otimes d}) \\
		&\geq& \chi(\overline{\mathcal B}^{\otimes d})\geq \frac{1}{2}d^2\overline{\mathcal B}. \overline{\mathcal B}+O(d\log d).
	\end{eqnarray*}
	Hence we have
	\begin{eqnarray*}
		\frac{\sum_{\sigma\in Z_d}e^{-\pi ||\sigma||^2}}{\exp(h_{\theta}^0(\X, \Li^{\otimes d}))}\leq \exp\left[ -2d^{1+\delta}+\frac{K}{\overline{\mathcal B} . \overline{\mathcal B}}d^{1+\beta} +O(d\log d) \right]\longrightarrow 0
	\end{eqnarray*}
	when $d \rightarrow \infty$, as by hypothesis $\delta>\beta$. This concludes our proof.
\end{proof}

\section{Arithmetic Bertini theorem on regularity}\label{Regular}

In this section, we prove Theorem \ref{mainvartheta}. In this proof, the main estimates on a single fiber, i.e. Proposition \ref{fiber} and Proposition \ref{singmed} below, has already been proved in \cite{Wa22}. Our main contribution here is to replace the restriction results on Arakelov density by similar results on $\theta$-density proved in Section \ref{restriction}.

\subsection{Singularities of small residual characteristic}
\begin{prop}\label{smalltheta}
	Let $\X$ be a regular projective arithmetic variety of absolute dimension $n$, and let $\Li$ be an ample Hermitian line bundle on $\X$. Set
	\[
		\mathcal{P}_{d, p\leq d^{\frac{1}{n+1}}}:=\left\{ \sigma\in \ho^0(\X,\Li^{\otimes d})\ ;\ \begin{array}{ll}
		\mathrm{div}\sigma \text{ has no singular point of residual} \\
		\text{characteristic at most } d^{\frac{1}{n+1}}
		\end{array}\right\}.
	\]
	When $d$ is sufficiently large, we have
	\[
		\left|  \frac{\sum_{\sigma\in\mathcal{P}_{d, p\leq d^{\frac{1}{n+1}}}}\exp(-\pi\lVert\sigma\rVert^2)}{ \sum_{\sigma \in \mathrm{H}^0(\X, \overline{\mathcal{L}}^{\otimes d})} \exp(-\pi \lVert\sigma\rVert^2) } -\zeta_\X(n+1)^{-1} \right| = O(d^{-\frac{1}{n+1} }).
	\]
	Here the constant involved in the big $O$ depends only on $\X$.
	
	In particular, writing $\mathcal{P}_B=\bigcup_{d> 0}\mathcal{P}_{d, p\leq d^{\frac{1}{n+1}}}$ we have
	\[
		\mu_{\theta}(\mathcal{P}_B)=\zeta_\X(n+1)^{-1}.
	\]
\end{prop}
We will use the estimate on one single fiber proved in \cite{Wa22} (Proposition 4.2 in that paper) :
\begin{prop}\label{fiber}
	Let $\X$ be a regular projective arithmetic variety of absolute dimension $n$, and let $\Li$ be an ample Hermitian line bundle on $\X$. For any prime number $p$, we define
	\[
		\mathcal{P}'_{d,p^2}:=\left\{ \sigma'\in \ho^0(\X_{p^2}, \overline{\mathcal{L}}^{\otimes d})\ ;\ \forall x\in |\mathrm{div}\sigma'|,\ \dim_{\kappa(x)}\frac{\m_{\mathrm{div}\sigma',x}}{\m^2_{\mathrm{div}\sigma', x}}=n-1 \right\},
	\]
	where $\X_{p^2}=\X\times_{\Z}\Z/p^2\Z$. There exists a constant $C>1$ such that for any large enough integer $d$ and any prime number $p$ verifying $Cnp^n<d$, we have
	\[
		\left| \frac{\#\mathcal{P'}_{d,p^2}}{\#\ho^0(\X_{p^2}, \li^{\otimes d})} -\zeta_{\X_p}(n+1)^{-1} \right|=O\left( \Big(\frac{d}{p}\Big)^{-\frac{2}{n}} \right),
	\] 
	where the constant involved in big $O$ is independent of $d,p$.
\end{prop}
\begin{rmq}
	Set
	\[
		\mathcal{P}_{d,p}:=\left\{ \sigma\in \ho^0(\X, \Li^{\otimes d})\ ;\ \mathrm{div}\ \sigma \text{ has no singular point on } \X_p \right\}.
	\]
	Then for any $\sigma\in \ho^0(\X, \Li^{\otimes d})$, we have $\sigma\in \mathcal{P}_{d,p}$ if and only if $\phi_{d,p^2}(\sigma)\in \mathcal{P'}_{d,p^2}$, in other words,
	\[
		\mathcal{P}_{d,p}=\phi_{d,p^2}^{-1}(\mathcal{P}'_{d,p^2}).
	\]
\end{rmq}

\begin{lem}\label{manytheta}
	With the same constant $C$ as in Proposition \ref{fiber}, for any large enough integer $d$ and any integer $r$ verifying $Cnr^n<d$ with $n=\dim \X$ and $N_r:=\prod_{p\leq r}p^2<e^{d^{\alpha_0}}$, we have
	\[
		\left|  \frac{\sum_{\sigma\in \left(\bigcap_{p\leq r}\mathcal{P}_{d,p}\right)} \exp(-\pi \lVert\sigma\rVert^2)}{ \sum_{\sigma \in \mathrm{H}^0(\X, \overline{\mathcal{L}}^{\otimes d})} \exp(-\pi \lVert\sigma\rVert^2) } -\prod_{p\leq r}\zeta_{\X_p}(n+1)^{-1}  \right| =O\left( \frac{\sum_{p\leq r}p^{\frac{2}{n}}}{d^{\frac{2}{n}}} \right).
	\]
\end{lem}
\begin{proof}
	Note that $\X_{N_r}=\coprod_{p\leq r}\X_{p^2}$. We have when $d$ is large enough,
	\begin{eqnarray*}
		\ho^0(\X_{N_r}, \li^{\otimes d}) \simeq  \prod_{p\leq r}\ho^0(\X_{p^2}, \li^{\otimes d}).
	\end{eqnarray*}
	Set
	\[
		E_{d,r}:=\{ \sigma\in \ho^0(\X_{N_r}, \li^{\otimes d})\ ;\ \forall x\in |\mathrm{div}\sigma|,\ \dim_{\kappa(x)}\frac{\m_{\mathrm{div}\sigma,x}}{\m^2_{\mathrm{div}\sigma, x}}=n-1 \}.
	\]
	Then we have
	\[
		\phi_{d, N_r}^{-1}(E_{d,r})=\bigcap_{p\leq r}\mathcal{P}_{d,p},
	\]
	where $\phi_{d,N_r}$ is the restriction map
	\[
		\phi_{d,N_r}: \ho^0(\X,\Li^{\otimes d}) \longrightarrow \ho^0(\X_{N_r},\li^{\otimes d}).
	\]
	On the other hand, still by the Chinese remainder theorem, 
	\[
		\frac{\#E_{d,r}}{\#\ho^0(\X_{N_r}, \li^{\otimes d})}=\prod_{p\leq r} \frac{\#\mathcal{P}'_{p^2,d}}{\#\ho^0(\X_{p^2}, \li^{\otimes d} )}.
	\]

	When $N_r<e^{d^{\alpha_0}}$, we can apply Corollary \ref{thetadensity} and get a constant $K_0>0$ such that
	\begin{eqnarray*}
		& &\left|  \frac{\sum_{\sigma\in \left(\bigcap_{p\leq r}\mathcal{P}_{d,p}\right)} \exp(-\pi \lVert\sigma\rVert^2)}{ \sum_{\sigma \in \mathrm{H}^0(\X, \overline{\mathcal{L}}^{\otimes d})} \exp(-\pi \lVert\sigma\rVert^2) } -\frac{\#E_{d,r}}{\#\ho^0(\X_{N_r}, \li^{\otimes d})}  \right| \\
		&<& 10 \exp \big(K_0d^{n-1}-\pi e^{2(\varepsilon_0-\delta) d}\big)\cdot \frac{\#E_{d,r}}{\#\ho^0(\X_{N_r}, \li^{\otimes d})} \\
		&\leq& 10 \exp \big(K_0d^{n-1}-\pi e^{2(\varepsilon_0-\delta) d}\big).
	\end{eqnarray*}
	Since
	\[
		\frac{\#E_{d,r}}{\#\ho^0(\X_{N_r}, \li^{\otimes d})}=\prod_{p\leq r} \frac{\#\mathcal{P}'_{p^2,d}}{\#\ho^0(\X_{p^2}, \li^{\otimes d} )},
	\]
	we get
	\[
		\left|  \frac{\sum_{\sigma\in \left(\bigcap_{p\leq r}\mathcal{P}_{d,p}\right)} \exp(-\pi \lVert\sigma\rVert^2)}{ \sum_{\sigma \in \mathrm{H}^0(\X, \overline{\mathcal{L}}^{\otimes d})} \exp(-\pi \lVert\sigma\rVert^2) } -\prod_{p\leq r} \frac{\#\mathcal{P}'_{p^2,d}}{\#\ho^0(\X_{p^2}, \li^{\otimes d} )}  \right| <10 \exp \big(K_0d^{n-1}-\pi e^{2(\varepsilon_0-\delta) d}\big).
	\]
	Note that we have by Lemma 5.3 in \cite{Wa22}
	\[
		\left| \prod_{p\leq r} \frac{\#\mathcal{P}'_{p^2,d}}{\#\ho^0(\X_{p^2}, \li^{\otimes d} )} - \prod_{p\leq r} \zeta_{\X_p}(1+n)^{-1} \right| =O\left( \frac{\sum_{p\leq r}p^{\frac{2}{n}}}{d^{\frac{2}{n}}} \right).
	\]
	The conclusion follows as
	\begin{eqnarray*}
		& &\left|  \frac{\sum_{\sigma\in \left(\bigcap_{p\leq r}\mathcal{P}_{d,p}\right)} \exp(-\pi \lVert\sigma\rVert^2)}{ \sum_{\sigma \in \mathrm{H}^0(\X, \overline{\mathcal{L}}^{\otimes d})} \exp(-\pi \lVert\sigma\rVert^2) } -\prod_{p\leq r}\zeta_{\X_p}(n+1)^{-1}  \right| \\
		&\leq& \left|  \frac{\sum_{\sigma\in \left(\bigcap_{p\leq r}\mathcal{P}_{d,p}\right)} \exp(-\pi \lVert\sigma\rVert^2)}{ \sum_{\sigma \in \mathrm{H}^0(\X, \overline{\mathcal{L}}^{\otimes d})} \exp(-\pi \lVert\sigma\rVert^2) } -\prod_{p\leq r} \frac{\#\mathcal{P}'_{p^2,d}}{\#\ho^0(\X_{p^2}, \li^{\otimes d} )}  \right| \\
		& & \qquad\qquad\qquad\qquad+\left| \prod_{p\leq r} \frac{\#\mathcal{P}'_{p^2,d}}{\#\ho^0(\X_{p^2}, \li^{\otimes d} )} - \prod_{p\leq r} \zeta_{\X_p}(1+n)^{-1} \right| \\
		&=& O\left(  \exp \big(K_0d^{n-1}-\pi e^{2(\varepsilon_0-\delta) d}\big) \right)+O\left( \frac{\sum_{p\leq r}p^{\frac{2}{n}}}{d^{\frac{2}{n}}} \right) \\
		&=& O\left( \frac{\sum_{p\leq r}p^{\frac{2}{n}}}{d^{\frac{2}{n}}} \right).
	\end{eqnarray*}
\end{proof}

Now we prove that we can choose $r=d^{\frac{1}{n+1}}$.
\begin{lem}\label{brtheta}
	For large enough integer $d$, we have
	\[
		\left|  \frac{\sum_{\sigma\in\mathcal{P}_{d, p\leq d^{\frac{1}{n+1}}}}\exp(-\pi \lVert\sigma\rVert^2)}{ \sum_{\sigma \in \mathrm{H}^0(\X, \overline{\mathcal{L}}^{\otimes d})} \exp(-\pi \lVert\sigma\rVert^2) } -\prod_{p\leq d^{\frac{1}{n+1}}}\zeta_{\X_p}(n+1)^{-1}  \right|=O(d^{-\frac{1}{n+1} }).
	\]
\end{lem}
\begin{proof}
	The proof follows that of Lemma 5.4 in \cite{Wa22}. 
	Whenever $2r\log r<d^{\alpha_0}$, we have
	\[
		N_r=\prod_{p\leq r}p^2\leq \prod_{k=1}^r k^2=(r!)^2<2 r^r=2\exp(r\log r)< \exp(d^{\alpha_0}).
	\]
	If moreover $r$ satisfies $Cnr^n<d$ for the constant $C$ given in Theorem \ref{fiber}, then by Lemma \ref{manytheta}, we have
	\[
		\left|  \frac{\sum_{\sigma\in \left(\bigcap_{p\leq r}\mathcal{P}_{d,p}\right)} \exp(-\pi \lVert\sigma\rVert^2)}{ \sum_{\sigma \in \mathrm{H}^0(\X, \overline{\mathcal{L}}^{\otimes d})} \exp(-\pi \lVert\sigma\rVert^2) } -\prod_{p\leq r}\zeta_{\X_p}(n+1)^{-1}  \right| =O\left( \frac{\sum_{p\leq r}p^{\frac{2}{n}}}{d^{\frac{2}{n}}} \right).
	\]
	
	Now as above, 
	\begin{eqnarray*}
		\sum_{p\leq r}p^{\frac{2}{n}} \leq \sum_{k=1}^r k^{\frac{2}{n}} < r\cdot  r^{\frac{2}{n}}= r^{\frac{n+2}{n}}.
	\end{eqnarray*}
	Thus for $r=d^{\frac{1}{n+1}}$ we have
	\[
		\sum_{p\leq d^{\frac{1}{n+1}}}p^{\frac{2}{n}}<d^{\frac{n+2}{n(n+1)}}=O(d^{\frac{n+2}{n(n+1)}}).
	\]
	It's easy to see that $r=d^{\frac{1}{n+1}}$ also satisfies conditions $2r\log r<d^{\alpha_0}$, $Cnr^n<d$ for large $d$. For this $r$, we have
	\[
		\mathcal{P}_{d, p\leq d^{\frac{1}{n+1}}}=\bigcap_{p\leq d^{\frac{1}{n+1}}}\mathcal{P}_{d,p},
	\]
	and hence
	\begin{eqnarray*}
		\left|  \frac{\sum_{\sigma\in\mathcal{P}_{d, p\leq d^{\frac{1}{n+1}}}}\exp(-\pi \lVert\sigma\rVert^2)}{ \sum_{\sigma \in \mathrm{H}^0(\X, \overline{\mathcal{L}}^{\otimes d})} \exp(-\pi \lVert\sigma\rVert^2) } -\prod_{p\leq d^{\frac{1}{n+1}}}\zeta_{\X_p}(n+1)^{-1}  \right| &=&O\left( \frac{\sum_{p\leq d^{\frac{1}{n+1}}}p^{\frac{2}{n}}}{d^{\frac{2}{n}}} \right)\\
		&=&O(d^{\frac{n+2}{n(n+1)}-\frac{2}{n}})=O(d^{-\frac{1}{n+1} }).
	\end{eqnarray*}
\end{proof}

\begin{proof}[Proof of Proposition \ref{smalltheta}]
	We know that (Lemma 3.4 in \cite{Wa22}) when $R\in \Z_{>0}$ is large enough, we have 
	\[
		\left| \prod_{p\leq R}\zeta_{\X_p}(n+1)^{-1}-\zeta_\X(n+1)^{-1} \right| < 8c_0 \zeta_\X(n+1)^{-1}\cdot R^{-1}. 
	\]
	Applying this result and Lemma \ref{brtheta}, for sufficiently large $d$, we have
	\begin{eqnarray*}
		& &\left|  \frac{\sum_{\sigma\in\mathcal{P}_{d, p\leq d^{\frac{1}{n+1}}}}\exp(-\pi \lVert\sigma\rVert^2)}{ \sum_{\sigma \in \mathrm{H}^0(X, \overline{\mathcal{L}}^{\otimes d})} \exp(-\pi \lVert\sigma\rVert^2) } -\zeta_\X(n+1)^{-1} \right| \\
		&\leq &\left|  \frac{\sum_{\sigma\in\mathcal{P}_{d, p\leq d^{\frac{1}{n+1}}}}\exp(-\pi \lVert\sigma\rVert^2)}{ \sum_{\sigma \in \mathrm{H}^0(X, \overline{\mathcal{L}}^{\otimes d})} \exp(-\pi \lVert\sigma\rVert^2) } -\prod_{p\leq d^{\frac{1}{n+1}}}\zeta_{\X_p}(n+1)^{-1}  \right|\\
		& &\qquad\qquad\qquad\qquad\qquad+\left| \prod_{p\leq d^{\frac{1}{n+1}} }\zeta_{\X_p}(n+1)^{-1}-\zeta_\X(n+1)^{-1} \right| \\
		&=&O(d^{-\frac{1}{n+1} })+O(d^{-\frac{1}{n+1} })=O(d^{-\frac{1}{n+1} }).
	\end{eqnarray*}
	Hence we conclude the proof of Proposition \ref{smalltheta}.
\end{proof}
\subsection{Singularities of large residual characteristic}
The following result is Proposition 6.1 in \cite{Wa22}.
\begin{prop}\label{singmed}
	Let $\X$ be a regular projective arithmetic variety of dimension $n$, and let $\li$ be an ample line bundle on $\X$. Then there exists a constant $c>0$ such that for any $d\gg0$ and any prime number $p$ such that $\X_p$ is smooth and irreducible, writing
	\[
		\mathcal{Q}_{d,p^2}:=\left\{ \sigma\in \ho^0(\X_{p^2}, \li^{\otimes d})\ ;\ \exists x\in |\X_{p^2}|,\ \dim_{\kappa(x)}\frac{\mathfrak{m}_{\mathrm{div}\sigma,x}}{\mathfrak{m}_{\mathrm{div}\sigma,x}^2}=n \right\},
	\]
	we have
	\[
		\frac{\#\mathcal{Q}_{d,p^2}}{\#\ho^0(\X_{p^2}, \li^{\otimes d})}\leq c\cdot p^{-2}.
	\]
\end{prop}

This proposition helps us to obtain the following result :
\begin{prop}\label{medsingtheta}
	Let $\X$ be a regular projective arithmetic variety of absolute dimension $n$, and let $\Li$ be an ample Hermitian line bundle on $\X$.  Let $\varepsilon_0$ be the constant found in Proposition \ref{epsilon}. Choose $0<\varepsilon<\frac{1}{2}\varepsilon_0$ and set 
	\[
		\mathcal{Q}_{d}^m:=\left\{ \sigma\in \ho^0(\X,\Li^{\otimes d})\ ;\ \begin{array}{ll}
		\mathrm{div}\sigma \text{ has a singular point of residual} \\
		\text{characteristic between } d^{\frac{1}{n+1}} \text{ and } e^{\varepsilon d}
		\end{array}\right\}.
	\]
	
	When $d$ is sufficiently large, we have
	\[
		\frac{\sum_{\sigma\in\mathcal{Q}_{d}^m}\exp(-\pi \lVert\sigma\rVert^2)}{ \sum_{\sigma \in \mathrm{H}^0(\X, \overline{\mathcal{L}}^{\otimes d})} \exp(-\pi \lVert\sigma\rVert^2) } = O(d^{-\frac{1}{n+1} }).
	\]
	Here the constant involved in the big $O$ depends only on $\X$.
	
	In particular, with $\mathcal{Q}^m=\bigcup_{d>0}\mathcal{Q}_{d}^m$, we have
	\[
		\mu_{\theta}(\mathcal{Q}^m)=0.
	\]
\end{prop}
\begin{proof}
	We have in fact
	\begin{eqnarray*}
		\frac{\sum_{\sigma\in\mathcal{Q}_{d}^m}\exp(-\pi \lVert\sigma\rVert^2)}{ \sum_{\sigma \in \mathrm{H}^0(\X, \overline{\mathcal{L}}^{\otimes d})} \exp(-\pi \lVert\sigma\rVert^2) }
		\leq \sum_{d^{\frac{1}{n+1}}\leq p \leq e^{\varepsilon d}} \frac{\sum_{\sigma\in \phi_{d, p^2}^{-1}\left( \mathcal{Q}_{d,p^2}\right) } \exp(-\pi \lVert\sigma\rVert^2) }{ \sum_{\sigma \in \mathrm{H}^0(\X, \overline{\mathcal{L}}^{\otimes d})} \exp(-\pi \lVert\sigma\rVert^2) }.
	\end{eqnarray*}
	Applying Corollary \ref{thetadensity} to the case $N=p^2$ and $E=\mathcal{Q}_{d,p^2}$, we obtain that there exists a constant $\delta>0$ such that when $d$ is large enough and $p^2\leq e^{\delta d}$,
	\begin{eqnarray*}
		 \frac{\sum_{\sigma\in \phi_{d, p^2}^{-1}\left( \mathcal{Q}_{d,p^2}\right) } \exp(-\pi \lVert\sigma\rVert^2) }{ \sum_{\sigma \in \mathrm{H}^0(\X, \overline{\mathcal{L}}^{\otimes d})} \exp(-\pi \lVert\sigma\rVert^2) }< 2\frac{\#\mathcal{Q}_{d,p^2} }{\#\mathrm{H}^0(\X_{p^2}, \li^{\otimes d})}.
	\end{eqnarray*}
	Since $\X$ is a regular arithmetic variety, it is irreducible and generically smooth. So if $d$ is large enough, for any prime number $p\geq d^{\frac{1}{n+1}}$, $\X_p$ is irreducible and smooth over $\F_p$. Then Proposition \ref{singmed} tells us that there exists a constant $c>0$ such that for any such $p$,
	\[
		\frac{\#\mathcal{Q}_{d,p^2} }{\#\mathrm{H}^0(\X_{p^2}, \li^{\otimes d})}\leq cp^{-2}.
	\]
	Setting $\varepsilon=\frac{\delta}{2}$, we get
	\begin{eqnarray*}
		\frac{\sum_{\sigma\in\mathcal{Q}_{d}^m}\exp(-\pi \lVert\sigma\rVert^2)}{ \sum_{\sigma \in \mathrm{H}^0(\X, \overline{\mathcal{L}}^{\otimes d})} \exp(-\pi \lVert\sigma\rVert^2) } &\leq& \sum_{d^{\frac{1}{n+1}}\leq p \leq e^{\varepsilon d}} \frac{\sum_{\sigma\in \phi_{d, p^2}^{-1}\left( \mathcal{Q}_{d,p^2}\right) } \exp(-\pi \lVert\sigma\rVert^2) }{ \sum_{\sigma \in \mathrm{H}^0(\X, \overline{\mathcal{L}}^{\otimes d})} \exp(-\pi \lVert\sigma\rVert^2) }\\
		&<&\sum_{d^{\frac{1}{n+1}}\leq p \leq e^{\varepsilon d}}2\frac{\#\mathcal{Q}_{d,p^2} }{\#\mathrm{H}^0(\X_{p^2}, \li^{\otimes d})} \\
		&\leq& \sum_{d^{\frac{1}{n+1}}\leq p \leq e^{\varepsilon d}} 2cp^{-2} \\
		&<& 2c d^{-\frac{1}{n+1}}.
	\end{eqnarray*}
	Hence we conclude.
\end{proof}

\subsection{Proof of Theorem \ref{mainvartheta}}
With the estimate proved above, we are now able to prove Theorem \ref{mainvartheta}.
\begin{proof}[Proof of Theorem \ref{mainvartheta}]
	As $\mathcal{P}_A\subset \mathcal{P}_B\subset \mathcal{P}_A\cup \mathcal{Q}^m$, we have 
	\begin{eqnarray*}
		& &\left|  \frac{\sum_{\sigma\in\mathcal{P}_{d,p\leq e^{\varepsilon d}}}\exp(-\pi \lVert\sigma\rVert^2)}{ \sum_{\sigma \in \mathrm{H}^0(\X, \overline{\mathcal{L}}^{\otimes d})} \exp(-\pi \lVert\sigma\rVert^2) } -  \frac{\sum_{\sigma\in\mathcal{P}_{d, p\leq d^{\frac{1}{n+1}}}}\exp(-\pi \lVert\sigma\rVert^2)}{ \sum_{\sigma \in \mathrm{H}^0(\X, \overline{\mathcal{L}}^{\otimes d})} \exp(-\pi \lVert\sigma\rVert^2) }  \right|\\
		&\leq& \frac{\sum_{\sigma\in\mathcal{Q}_{d}^m}\exp(-\pi \lVert\sigma\rVert^2)}{ \sum_{\sigma \in \mathrm{H}^0(\X, \overline{\mathcal{L}}^{\otimes d})} \exp(-\pi \lVert\sigma\rVert^2) }\\
		&=&O(d^{-\frac{1}{n+1} })
	\end{eqnarray*}
	by Proposition \ref{medsingtheta}, and so
	\begin{eqnarray*}
		& &\left| \frac{\sum_{\sigma\in\mathcal{P}_{d,p\leq e^{\varepsilon d}}}\exp(-\pi \lVert\sigma\rVert^2)}{ \sum_{\sigma \in \mathrm{H}^0(\X, \overline{\mathcal{L}}^{\otimes d})} \exp(-\pi \lVert\sigma\rVert^2) } -\zeta_\X(n+1)^{-1} \right|\\
		&=&\left|  \frac{\sum_{\sigma\in\mathcal{P}_{d,p\leq e^{\varepsilon d}}}\exp(-\pi \lVert\sigma\rVert^2)}{ \sum_{\sigma \in \mathrm{H}^0(\X, \overline{\mathcal{L}}^{\otimes d})} \exp(-\pi \lVert\sigma\rVert^2) } -  \frac{\sum_{\sigma\in\mathcal{P}_{d, p\leq d^{\frac{1}{n+1}}}}\exp(-\pi \lVert\sigma\rVert^2)}{ \sum_{\sigma \in \mathrm{H}^0(\X, \overline{\mathcal{L}}^{\otimes d})} \exp(-\pi \lVert\sigma\rVert^2) }  \right|\\
		& & \qquad\qquad+\left|  \frac{\sum_{\sigma\in\mathcal{P}_{d, p\leq d^{\frac{1}{n+1}}}}\exp(-\pi\lVert\sigma\rVert^2)}{ \sum_{\sigma \in \mathrm{H}^0(\X, \overline{\mathcal{L}}^{\otimes d})} \exp(-\pi \lVert\sigma\rVert^2) } -\zeta_\X(n+1)^{-1} \right|\\ 
		&=& O(d^{-\frac{1}{n+1} })+O(d^{-\frac{1}{n+1} })=O(d^{-\frac{1}{n+1} }).
	\end{eqnarray*}
	This implies the conclusion
	\[
		\mu_{\theta}(\mathcal{P}_A)=\mu_{\theta}(\mathcal{P}_B)=\zeta_{\X}(1+n)^{-1}.
	\]	
\end{proof}

\end{document}